\documentclass[reqno,11pt]{amsart}
 
\usepackage{amsmath}
\usepackage{amsthm,amssymb,color,comment}
\usepackage{bbm}
\usepackage{mathrsfs}
\usepackage{hyperref}
\usepackage[TS1,T1]{fontenc}
\usepackage[utf8]{inputenc}
\usepackage{dsfont}
\usepackage{tikz}
\usepackage{enumitem}
\usepackage{caption}
\usepackage{subcaption}
\usepackage{mathtools}
\usepackage{bbold}
\usepackage{esint}
\usepackage[sort]{cite}%
\usepackage{float}
\usepackage{tikz}
\usetikzlibrary{shapes.geometric, calc, arrows.meta}

\numberwithin{equation}{section}

\newtheorem{thm}{Theorem}[section]
\newtheorem{proposition}[thm]{Proposition}

\definecolor{ballblue}{rgb}{0.13, 0.67, 0.8}

\theoremstyle{definition}

\newtheorem{rem}[thm]{Remark}

\newcommand{\R}{\mathbb{R}}

\newcommand{\p}{\partial}
\newcommand {\f}{\frac}
\newcommand{\eps}{\varepsilon}

\newcommand{\diff}{\mathop{}\!\mathrm{d}}

\newcommand{\ueps}{u_{\varepsilon}}

\newcommand{\barueps}{\bar{u}_{\varepsilon}}
\newcommand{\Weps}{W_{\varepsilon}}
\newcommand{\Dx}{\Delta x}
\newcommand{\dy}{\textrm{d}y}
\newcommand{\uepsbar}{\overline{u}_{\varepsilon}}
\newcommand{\dx}{\textrm{d}x}

\makeatletter
\newcommand{\doublewidetilde}[1]{{%
  \mathpalette\double@widetilde{#1}%
}}
\newcommand{\double@widetilde}[2]{%
  \sbox\z@{$\m@th#1\widetilde{#2}$}%
  \ht\z@=.9\ht\z@
  \widetilde{\box\z@}%
}
\makeatother
\author{José A. Carrillo}
\address{{\it José A. Carrillo:} Mathematical Institute, University of Oxford, Woodstock Road, Oxford, OX2 6GG, United Kingdom}
\email{carrillo@maths.ox.ac.uk}

\author{Charles Elbar}
\address{{\it Charles Elbar:} Université Claude Bernard Lyon 1, ICJ UMR5208, CNRS, Ecole Centrale de Lyon, INSA Lyon, Université Jean Monnet, 69622
Villeurbanne, France}
\email{elbar@math.univ-lyon1.fr}
\thanks{}

\author{Stefano Fronzoni}
\address{{\it Stefano Fronzoni:} Mathematical Institute, University of Oxford, Woodstock Road, Oxford, OX2 6GG, United Kingdom}
\email{fronzoni@maths.ox.ac.uk}

\author{Jakub Skrzeczkowski}
\address{{\it Jakub Skrzeczkowski: } Mathematical Institute, University of Oxford, Woodstock Road, Oxford, OX2 6GG, United Kingdom}
\email{jakub.skrzeczkowski@maths.ox.ac.uk}

\begin{document}

\title[Rate of convergence for a nonlocal-to-local limit in one dimension]{Rate of convergence for a nonlocal-to-local limit in one dimension}

\begin{abstract}
We consider a nonlocal approximation of the quadratic porous medium equation where the pressure is given by a convolution with a mollification kernel. It is known that when the kernel concentrates around the origin, the nonlocal equation converges to the local one. In one spatial dimension, for a particular choice of the kernel, and under mere assumptions on the initial condition, we quantify the rate of convergence in the 2-Wasserstein distance. Our proof is very simple, exploiting the so-called Evolutionary Variational Inequality for both the nonlocal and local equations as well as a priori estimates. We also present numerical simulations using the finite volume method, which suggests that the obtained rate can be improved - this will be addressed in a forthcoming work.
\end{abstract}



\maketitle
\setcounter{tocdepth}{1}

\section{Introduction}

We consider the nonlocal PDE 
\begin{equation}\label{eq:nonlocal_PDE_one_dimension}
\p_t u_\eps  - \p_x(u_\eps\,  \p_x W_{\eps}\ast u_\eps)=0
\end{equation}
where $W_{\eps} \approx \delta_0$ as $\eps \to 0$. More precisely, we choose $W_{\eps}$ to satisfy $W_{\eps}(x) = \f{1}{\eps} W\left(\f{x}{\eps}\right)$ where $W(x) =\f{1}{2} e^{-|x|}$. It is well-known that $W_{\eps}$ satisfies the elliptic PDE
\begin{equation}\label{eq:elliptic_PDE}
\quad -\eps^2\p_{xx}W_\eps + W_\eps = \delta_{0}.
\end{equation}
It is known \cite{david2023degenerate, elbar2023inviscid} that solutions of \eqref{eq:nonlocal_PDE_one_dimension} converge as $\eps \to 0$ to the solution of
\begin{equation}\label{eq:local_PDE}
\p_tu - \p_x(u \, \p_x u) = \p_t u - \p_{xx}\f{u^2}{2}=0. 
\end{equation}

The nonlocal-to-local limit as above appears in several contexts. In mathematical biology, it connects models of tissue growth with so-called Brinkman's and Darcy's pressures as given by \eqref{eq:nonlocal_PDE_one_dimension} and \eqref{eq:local_PDE} respectively \cite{elbar2023inviscid, david2023degenerate}. This type of model has been studied in the context of their connection with the free boundary model of tumor growth \cite{MR3695889, MR3162474, MR4188329, MR4324293}. In numerical analysis, the nonlocal equation \eqref{eq:nonlocal_PDE_one_dimension}, viewed as a continuity equation, provides an approximation scheme, called the blob method, for solving \eqref{eq:local_PDE}. Starting from works of Oelschl\"ager \cite{Ol} and Lions-MasGallic \cite{MR1821479}, this approach has been extensively studied and generalized to more complicated equations and systems, including nonlinear porous medium equations \cite{carrillo2023nonlocal, MR3913840}, cross-diffusion systems \cite{Hecht2023porous, burger2022porous}, heat equation and fast-diffusion equations \cite{MR4858611, carrillo2024nonlocal}. \\\

The target of this paper is to provide a simple argument leading to the rate of convergence as $\eps \to 0$ in the Wasserstein distance. In what follows, we write $\mathcal{P}_2(\R)$ for the space of probability measures $\rho$ such that $\int_{\R} \rho\, |x|^2 \diff x < \infty$ and $\mathcal{W}_2$ for the second Wasserstein distance.

\begin{thm}\label{thm:bessel_rate_eps}

Let $u^0 \in \mathcal{P}_2(\R)\cap L^{\infty}(\R)$. Let $u_{\eps}$ be the unique solution of 
\eqref{eq:nonlocal_PDE_one_dimension} and let $u$ be the unique solution of \eqref{eq:local_PDE} with initial condition $u^0$. Then, there exists a constant $C = C(u^0, T, W)$ such that for all $t \in [0,T]$
\vspace{-0.2em}
\begin{equation} \label{eq:bessel_rate_eps}
\mathcal{W}_2(u_{\eps}(t,\cdot), u(t,\cdot)) \leq C\, \sqrt{\eps}.
\end{equation}
\end{thm}

We present the proof in Section \ref{sect:proof_main_thm}. The two main ingredients are a priori estimates for \eqref{eq:nonlocal_PDE_one_dimension} and evolutionary variational inequality for both \eqref{eq:nonlocal_PDE_one_dimension} and \eqref{eq:local_PDE}. In Section \ref{sect:numerics} we present numerical simulations obtained via finite volume method. \\

The recent paper \cite{amassad2025deterministic} discusses the rate of convergence of the nonlocal-to-local limit in much broader generality. First, in one space dimension they are able to consider any kernel which is convex on $(-\infty,0)$ and $(0,\infty)$ including our result. Second, they are able to establish the rate in higher dimensions. While it is true that \cite{amassad2025deterministic} includes our result, the methods are substantially different and more complicated. The result of \cite{amassad2025deterministic} is based on the well-known formula for the time derivative of the Wasserstein distance \cite[Theorem 5.24]{ref:Santambrogio2015} while our work is based on the EVI formulation. The difference is also reflected by the assumptions on the initial datum: \cite{amassad2025deterministic} needs $u^0$ to be in $W^{1,\infty}(\mathbb{T})$ (the paper discusses the problem on the torus $\mathbb{T}$) while we require $L^1(\R) \cap L^{\infty}(\R)$ regularity (together with $u^0\, |x|^2 \in L^1(\R)$). In particular, we do not make any assumptions on the derivatives of $u^0$.\\

It is natural to ask whether the rate $\sqrt{\varepsilon}$, also obtained in \cite{amassad2025deterministic}, is optimal. The numerical simulations presented in Figures \ref{Fig:GaussConv} and \ref{Fig:ParabConvPBC} suggest that the rate could potentially be improved to $\varepsilon$ for small values of $\varepsilon$. In a forthcoming work \cite{art_new_formula_Wasserstein}, we confirm that this is indeed the case. The result is based on a new formula for the Wasserstein distance between two gradient flows and on Aronson-Bénilan-type estimates for the porous medium equation.

\section{Proof of Theorem~\ref{thm:bessel_rate_eps}}\label{sect:proof_main_thm}

\subsection{Uniform estimates for nonlocal equation \eqref{eq:nonlocal_PDE_one_dimension}}

In this section, for the sake of completeness, we briefly recall the following uniform bounds on \eqref{eq:nonlocal_PDE_one_dimension} obtained in \cite{Perthame-incompressible-visco,MR3260280}. 

\begin{proposition}\label{prop:apriori_estimates}
Let $u^0 \in L^1(\R)\cap L^{\infty}(\R)$, $u^0 \geq 0$, $u^0 |x|^2 \in L^1(\R)$. Then, there exists unique solution $\{u_{\eps}\}$ to \eqref{eq:nonlocal_PDE_one_dimension} with initial condition $u^0$ such that 
\begin{enumerate}[label=(\Alph*)]
    \item $\{u_{\eps}\}$ is uniformly bounded in $L^p((0,T)\times\R)$ for all $p\in [1,\infty]$,
    \item $\{u_{\eps}\, |\log(u_{\eps})|\}$, $\{u_{\eps}\,|x|^2\}$ are uniformly bounded in $L^{\infty}(0,T; L^1(\R))$,
    \item\label{item:estim3} $\{\p_x W_{\eps} \ast u_{\eps}\}$ is uniformly bounded in $L^2((0,T)\times\R)$,
    \item\label{item:estim4} $\{\eps\,\p_{xx} W_{\eps} \ast u_{\eps}\}$ is uniformly bounded in $L^2((0,T)\times\R)$,
    \item $\{\sqrt{u_{\eps}}\, |\p_x W_{\eps}\ast u_{\eps}|^2\}$ is uniformly bounded in $L^2((0,T)\times\R)$.
\end{enumerate}
\end{proposition}

\begin{rem}\label{rem:control_of_log}
The assumptions on the initial datum $u^0$ imply that $\int_{\R} u^0\, |\log u^0| \diff x < \infty$. Indeed, one splits $\R$ for three sets $A_1=\{x \in \R: u^0>1\}$, $A_2=\{x \in \R: u^0 \leq e^{-|x|^2}\}$, $A_3=\{x \in \R: e^{-|x|^2} < u^0 \leq 1\}$ so that
$$
\int_{\R} u^0\, |\log u^0| \diff x \leq \int_{A_1} u^0\, (u^0 + 1) \diff x + \int_{A_2} e^{-|x|^2/2} \diff x + \int_{A_3} u^0\, |x|^2 \diff x.
$$
On the set $A_2$ we used that $|\sqrt x \, \log x|\leq 1 $ for $0\leq x\leq 1$. 
\end{rem}
\begin{proof}
The existence and uniqueness in one dimension is classical and can be proved either directly by a fixed point argument as in \cite{MR4146915} or by using geodesic convexity of the energy $E_{\eps}[\rho] = \int_{\R} \rho\ast W_{\eps} \, \rho \diff x$ on $\mathcal{P}_2(\R)$ equipped with the Wasserstein distance $\mathcal{W}_2$ cf. \cite[Proposition~11.1.4]{MR2401600}. Concerning the estimates, we first write \eqref{eq:nonlocal_PDE_one_dimension} as 
$$
\p_t u_{\eps} - \p_x u_{\eps} \, \p_x W_{\eps} \ast u_\eps - u_{\eps} \, \p_{xx} W_{\eps} \ast u_\eps = 0
$$
so that using \eqref{eq:elliptic_PDE} we get
$$
\p_t u_{\eps} =  \p_x u_{\eps} \, \p_x W_{\eps} \ast u_\eps - \frac{1}{\eps^2}( u_{\eps} - W_{\eps} \ast u_\eps). 
$$
Noting that in the point $x^* \in \R$ where $x \mapsto u_{\eps}(t,x)$ reaches maximum we have $\p_x u_{\eps}(t,x^*) = 0$ and $ u_{\eps} - W_{\eps} \ast u_\eps 
\geq 0$ we obtain $\p_t u_{\eps}(x^*) \leq 0$ and this proves the claim. This argument comes from \cite[Lemma 2.1]{MR3037041}.\\

The other estimates follow by usual energy considerations. Firstly, we multiply~\eqref{eq:nonlocal_PDE_one_dimension} by $W_{\eps}\ast u_{\eps}$ and we obtain 
\begin{equation}\label{eq:energy_dissipation_identity}
\partial_t \int_{\R^d} \f{1}{2}u_{\eps} W_{\eps}\ast u_{\eps} \diff x + \int_{\R} u_{\eps} |\p_{x}(W_{\eps}\ast u_{\eps})|^2\diff x.
\end{equation}
Next, we multiply \eqref{eq:nonlocal_PDE_one_dimension} with $|x|^2/2$ to get
\begin{align*}
\partial_t\,\frac{1}{2} \int_{\R} u_{\eps} |x|^2 \diff x = - \int_{\R} u_{\eps}\, \p_x (W_{\eps}\ast u_{\eps} )\, x \diff x \leq \left(\int_{\R} u_\eps |x|^2 \diff x \right)^{1/2} \left(\int_{\R}u_{\eps}|\p_x W_{\eps}\ast u_{\eps}|^2 \diff x\right)^{1/2},   
\end{align*}
so that using \eqref{eq:energy_dissipation_identity} and the Gronwall's inequality, we obtain a uniform bound on $\int_{\R} u_{\eps}\, |x|^2 \diff x$. Finally, multiplying \eqref{eq:nonlocal_PDE_one_dimension} with $\log u_{\eps}$ and using \eqref{eq:elliptic_PDE} we get 
$$
\partial_t \int_{\R} u_{\eps} \log u_{\eps} \diff x + \int_{\R} |\p_x W_{\eps} \ast u_{\eps}|^2 \diff x + \eps^2 \, \int_{\R} |\p_{xx} W_\eps \ast u_\eps|^2 \diff x  = 0.
$$
The negative part $\int_{\R} u_{\eps} |\log u_{\eps}|^{-} \diff x$ is easily controlled by means of the moment estimate $\int_{\R} u_{\eps}\,|x|^2 \diff x$ as in Remark \ref{rem:control_of_log}. The proof of Proposition \ref{prop:apriori_estimates} is concluded. 
\end{proof}

\subsection{The theory of gradient flows}

We recall from \cite[Theorem 11.1.4]{MR2401600} the following result characterizing gradient flows on the space of probability measures $\mathcal{P}_2(\R)$ equipped with the Wasserstein distance $\mathcal{W}_2$.

\begin{thm}\label{thm:gradient_flows}
Let $F: \mathcal{P}_2(\R) \to (-\infty, +\infty]$ be a lower semicontinuous and $\lambda$-geodesically convex functional (on $(\mathcal{P}_2(\R), \mathcal{W}_2)$). For any $\mu_0 \in \mathcal{P}_2(\R)$ there exists at most one gradient flow $\mu_t$ satisfying the initial condition $\mu_0$ characterized by the evolutionary variational inequality 
\begin{equation}\label{eq:EVI_general}
F(\mu_t) + \frac{1}{2} \frac{\diff}{\diff t} \mathcal{W}_2^2(\mu_t, \sigma) + \frac{\lambda}{2} \, \mathcal{W}_2^2(\mu_t, \sigma) \leq F(\sigma)
\end{equation}
which holds for a.e. $t$ and $\sigma \in D(F)$ (the domain of $F$).
\end{thm}

In the proof of Theorem \ref{thm:bessel_rate_eps} we apply Theorem \ref{thm:gradient_flows} to two functionals
$$
E[\rho] = \frac{1}{2} \int_{\R} \rho^2 \diff x, \qquad \qquad E_{\eps}[\rho] =  \frac{1}{2} \int_{\R} W_{\eps}\ast \rho \, \rho \diff x.
$$
\begin{proposition}
The functionals $E$ and $E_{\eps}$ are $\lambda$-geodesically convex on $(\mathcal{P}_2(\R), \mathcal{W}_2)$ with $\lambda = 0$.
\end{proposition}
\begin{proof}
The functional $E$ is well known to be geodesically convex cf. \cite[Theorem 5.15]{MR1964483}. Concerning $E_{\eps}$, we follow the argument from \cite[Proposition 2.7]{MR2935390}. Consider $\rho_0$, $\rho_1\in \mathcal{P}_{2}(\R)$ and the geodesic $\rho_t = ((1-t) \, I + t \, T)\# \rho_0$ where $I$ is the identity map while $T$ is the optimal transport map between $\rho_0$ and $\rho_1$. We compute
$$
E_{\eps}[\rho_t] = \frac{1}{2} \int_{\R\times\R}W_\eps ((1-t)(x-y) + t(T(x)-T(y)) \diff \rho_0 (x) \diff \rho_0(y)
$$
Then we can split the integral for the sets $x>y$ and $x\leq y$. Since the transport map $T$ in one dimension is monotone (so that $x > y$ implies $T(x) \geq T(y)$) and $W$ is convex in $(-\infty,0]$ and $[0,+\infty)$ we compute
\begin{align*}
E_{\eps}[\rho_t]&\le \frac{(1-t)}{2} \int_{\R}\int_{x>y}W_{\eps}(x-y)\diff\rho_{0}(x)\diff\rho_{0}(y) + \frac{t}{2} \int_{\R}\int_{x>y}W_{\eps}(x-y)\diff\rho_{1}(x)\diff\rho_{1}(y)\\
&+ \frac{(1-t)}{2} \int_{\R}\int_{x\leq y}W_{\eps}(x-y)\diff\rho_{0}(x)\diff\rho_{0}(y) + \frac{t}{2} \int_{\R}\int_{x \leq y}W_{\eps}(x-y)\diff\rho_{1}(x)\diff\rho_{1}(y)\\
&=(1-t)E_{\eps}[\rho_0] + t E_{\eps}[\rho_1],
\end{align*}
so that $E_{\eps}$ is indeed geodesically convex.
\end{proof}

\subsection{Proof of Theorem~\ref{thm:bessel_rate_eps}}

The unique solutions of the EVI \eqref{eq:EVI_general} in Theorem \ref{thm:gradient_flows} with $F= E_{\eps}$ and $F=E$ coincide with the unique solutions of $u_{\eps}$, $u$ of \eqref{eq:nonlocal_PDE_one_dimension} and \eqref{eq:local_PDE}, see for instance~\cite[Proposition 3.28, Proposition 3.38]{MR3050280}.
By the EVI \eqref{eq:EVI_general} for \eqref{eq:nonlocal_PDE_one_dimension} we have
\begin{equation}\label{eq:EVI_eps_applied}
\frac{1}{2} \, \frac{\diff}{\diff t} \mathcal{W}_2^2(u_{\eps}(t,\cdot), u(s,\cdot)) \leq E_{\eps}(u(s,\cdot)) - E_{\eps}(u_{\eps}(t,\cdot)).
\end{equation}
Similarly, by the EVI \eqref{eq:EVI_general} for \eqref{eq:local_PDE}
\begin{equation}\label{eq:EVI_applied}
\frac{1}{2} \, \frac{\diff}{\diff t} \mathcal{W}_2^2(u(t,\cdot), u_{\eps}(s,\cdot)) \leq E(u_{\eps}(s,\cdot)) - E(u(t,\cdot)).
\end{equation}
Combining \eqref{eq:EVI_eps_applied} and \eqref{eq:EVI_applied} we get
$$
\frac{1}{2} \, \frac{\diff}{\diff t} \mathcal{W}_2^2(u_{\eps}(t,\cdot), u(t,\cdot)) \leq  E_{\eps}(u(t,\cdot)) -  E(u(t,\cdot)) + E(u_{\eps}(t,\cdot)) - E_{\eps}(u_{\eps}(t,\cdot)).  
$$
By the inequality $u(x)\,u(y) \leq \frac{1}{2}  u(x)^2 + \frac{1}{2} u(y)^2$ we have $E_{\eps}(u(t,\cdot)) -  E(u(t,\cdot)) \leq 0$. Therefore, it is sufficient to estimate $\int_0^T \left|E_{\eps}(u_{\eps}(t,\cdot)) -  E(u_{\eps}(t,\cdot))\right| \diff t$. Clearly,
$$
\left|E_{\eps}(u_{\eps}(t,\cdot)) -  E(u_{\eps}(t,\cdot))\right| = \frac{1}{2} \left| \int_{\R} u_{\eps}\left(u_{\eps} - u_{\eps}\ast W_{\eps} \right) \diff x \right| \leq \frac{1}{2} \|u_{\eps}\|_{L^2_x} \, \|u_{\eps} - u_{\eps}\ast W_{\eps}\|_{L^2_x}. 
$$
Using the elliptic PDE \eqref{eq:elliptic_PDE}, we write $u_{\eps} - u_{\eps}\ast W_{\eps} = -\eps^2 \partial_{xx} u_{\eps}\ast W_{\eps}$ so that integrating in time and using \ref{item:estim4} in Proposition \ref{prop:apriori_estimates} 
$$
\int_0^T \left|E_{\eps}(u_{\eps}(t,\cdot)) -  E(u_{\eps}(t,\cdot))\right| \diff t \leq \frac{1}{2} \|u_{\eps}\|_{L^2_{t,x}} \, \|\eps^2 \partial_{xx} u_{\eps}\ast W_{\eps}\|_{L^2_{t,x}} \leq C\, \eps.
$$
The proof is concluded.

\section{A Finite Volume scheme}\label{sect:numerics}

We propose now a numerical scheme in one dimension, based on finite volumes, to analyse further the result of Theorem \ref{thm:bessel_rate_eps}. We consider a one-dimensional domain $\Omega = (r_{1}, r_{2})$ and we divide it into $N$ cells of the form $[x_{i-\frac{1}{2}}, x_{i+\frac{1}{2}} ]$, for $i=1, \dots, N$, and for simplicity we assume that the cells have uniform size, that is $x_{i+\frac{1}{2}} - x_{i-\frac{1}{2}} = (r_2 - r_1)/N =: \Delta x$ for all $i = 1, \dots, N$. Each cell is centered at the points $x_{i}$, where $x_{i} = r_1 + (i/2) \Delta x$. 

We denote by $\barueps^i(t)$ the average of the solution $\ueps(x,t)$ over the $i$-th cell:
\[ \barueps^i(t) = \frac{1}{\Delta x} \int_{C_i} \ueps(t, x)\, \dx. \]

If we now integrate equation \eqref{eq:nonlocal_PDE_one_dimension} on each cell $C_i$ we have a system of differential equations, given by 
\begin{subequations} \label{eqs:FVschem}
\begin{equation}
\frac{d \barueps^{i}}{dt} = \frac{F(\ueps(t, x_{i+\frac{1}{2}})) - F(\ueps(t, x_{i-\frac{1}{2}}))}{\Delta x}, \quad i=1, \dots, N
\end{equation}

where the flux $F$ is given by $F(\ueps):=\ueps\, \partial_x \Weps \ast \ueps$. 

We then approximate the flux at the cell interface with functions $F_{i + \frac{1}{2}}$ and $F_{i-\frac{1}{2}}$ approximating respectively $F(\ueps(t, x_{i+\frac{1}{2}}))$ and  $F(\ueps(t, x_{i-\frac{1}{2}}))$. For constructing these approximations we use upwinding in a standard way for which the reader can refer to \cite{CarrilloChertock2015, carrillofronzoni2024}: 
\begin{equation} \label{eq:upwinding}
F_{i+\frac{1}{2}} = v_{i+\frac{1}{2}}^{+} (\ueps^i)^{\textrm{R}} + v_{i-\frac{1}{2}}^{-} (\ueps^i)^{\textrm{L}},
\end{equation}
\end{subequations}
where 
\begin{equation*}
    v_{i+\frac{1}{2}} =  \frac{\Weps \ast \ueps(x_{i+1}) - \Weps \ast \ueps(x_{i})}{\Delta x}, \quad v_{i+\frac{1}{2}}^{+}= \max(v_{i+\frac{1}{2}}, 0), \quad v_{i+\frac{1}{2}}^{-} = \min(v_{i+\frac{1}{2}}, 0). 
\end{equation*}
A corresponding expression holds for $F_{i-\frac{1}{2}}$.

For the values of $(\ueps^i)^{\mathrm{R}}$ and $(\ueps^i)^{\mathrm{L}}$ we can choose a first or second order approximation. We can define
\[ \tilde{\ueps}(x) = \barueps^i + (\ueps)_{x}^{i} (x - x_i) \quad \text{ for } x \in [x_{i-\frac{1}{2}}, x_{i+\frac{1}{2}}] \]
and we would have 
\[ 
(\ueps^i)^{\mathrm{R}} = \barueps^i + \frac{\Delta x}{2} (\ueps)_{x}^{i}, \qquad (\ueps^i)^{\mathrm{L}} = \barueps^i - \frac{\Delta x}{2} (\ueps)_{x}^{i}.  
\]
If one sets $(\ueps)_{x}^{i}=0$ for all $i=1, \dots, N$, we have a first order approximation, while if we consider a piecewise linear approximation, the scheme is  second order accurate in space by imposing nonnegativity of the density reconstruction. With this purpose, we can make a choice for a non-vanishing numerical slope
limiter guaranteeing that the reconstructed point values are nonnegative as long as the values of the solution are nonnegative. We can use a generalized minmod limiter
\[ 
(\ueps)_{x}^{i} = \text{minmod} \bigg( \theta \frac{\barueps^{i+1} -\barueps^i}{\Delta x}, \frac{\barueps^{i+1} -\barueps^{i-1}}{2 \Delta x}, \theta \frac{\barueps^{i} -\barueps^{i-1}}{\Delta x} \bigg), \]
where the parameter $\theta$ controls numerical viscosity and 
\begin{align}
	\textrm{minmod}(a,b,c) \coloneqq
	\begin{cases}
		\min \{a,b,c\} & \text{if }a,b,c>0, \\
		\max \{a,b,c\} & \text{if }a,b,c<0, \\
		0                & \text{otherwise}.
	\end{cases}
\end{align}
The term $\Weps \ast \ueps(x_{i})$ is approximated as a linear combination
\begin{equation}
\Weps \ast \ueps(x_{i}) = \sum_{j=1}^{N} \bigg( \int_{x_{j-\frac{1}{2}}}^{x_{j+\frac{1}{2}}} \frac{1}{2\varepsilon} e^{-\frac{|x_{i} - y|}{\varepsilon}} \dy\bigg) \uepsbar^{j} = (\mathcal{I} \uepsbar)_{i},
\end{equation}
where $\mathcal{I} \in \mathbb{R}^{N \times N}$ is the matrix
\begin{equation}
\mathcal{I} = \begin{pmatrix} I_{1}(x_{1}) & I_{2}(x_{1}) & \hdots & I_{N}(x_{1}) \\ I_{1}(x_{2}) & I_{2}(x_{2}) & \hdots & I_{N}(x_{2}) \\ \vdots & \vdots & \ddots & \vdots \\ I_{1}(x_{N}) & I_{2}(x_{N}) & \hdots & I_{N}(x_{N}) \end{pmatrix},
\end{equation}
with 
\begin{equation}
I_{j}(x_{i}) = \int_{x_{j-\frac{1}{2}}}^{x_{j+\frac{1}{2}}} \frac{1}{2\varepsilon} e^{-\frac{|x_{i} - y|}{\varepsilon}} \dy.
\end{equation}
Moreover, the above term can be explicitly computed as
\begin{equation*}
I_{j}(x_{i}) = \left \{ \begin{array}{ll} \frac{1}{2} \Big( e^{\frac{x_{i} - x_{j-\frac{1}{2}}}{\varepsilon}} - e^{\frac{x_{i} - x_{j+\frac{1}{2}}}{\varepsilon}} \Big) & \text{if } x_{j-\frac{1}{2}} > x_{i}, \\ \frac{1}{2} \Big( e^{-\frac{x_{i} - x_{j+\frac{1}{2}}}{\varepsilon}} - e^{-\frac{x_{i} - x_{j-\frac{1}{2}}}{\varepsilon}} \Big) & \text{if } x_{j+\frac{1}{2}} < x_{i}, \\ 1 - e^{-\frac{\Dx}{2 \varepsilon}} & \text{if } i = j. \end{array} \right.
\end{equation*}

\begin{rem}
Notice that for fixed $\Delta x$ as $\varepsilon \to 0$ the matrix $\mathcal{I}$ tends to the identity matrix, making the scheme converging to the scheme of the standard porous medium equation with exponent 2. Therefore, from the numerical point of view, we recover the switch from nonlocal (for larger values of $\varepsilon$ giving a full matrix $\mathcal{I}$) to local (for $\varepsilon$ approaching zero, the sparsity of $\mathcal{I}$ increases up to reaching the identity). 
\end{rem}

\begin{rem}
The scheme \eqref{eqs:FVschem} has to be paired with some boundary conditions. In our experiments we will use either a no-flux boundary condition or a periodic boundary condition.
The no-flux boundary condition imposes a vanishing flux at the endpoints of the one-dimensional domain
    $F_{i-\frac{1}{2}} = F_{N+\frac{1}{2}} = 0$, while the periodic boundary condition imposes $F_{i-\frac{1}{2}} = F_{N+\frac{1}{2}}$. We will see later that the choice of the type of boundary condition is crucial for the rate of convergence in $\varepsilon$. 
\end{rem}

\begin{rem}
    In our numerical experiments, in addition to the finite volume scheme, that we have summarised above for equation \eqref{eq:nonlocal_PDE_one_dimension}, we will also consider a finite volume scheme to the self-similar Fokker-Planck equation, associated to it. In particular we will consider the Fokker-Planck equation 
    \begin{equation}\label{eq:local_PDE_FoPla}
\p_tu - \p_x(u \, \p_x u) - \partial_x(x u) = 0,
\end{equation}
    together with its nonlocal approximation, given by 
    \begin{equation} \label{eq:nonlocal_PDE_one_dimension_FoPla}
        \partial_t \ueps - \partial_x (\ueps \partial_x W_\eps \ast \ueps) - \partial_x(x \ueps) = 0.
    \end{equation}
    The numerical scheme for \eqref{eq:nonlocal_PDE_one_dimension_FoPla} can be easily written as \eqref{eqs:FVschem}, where the flux terms $F_{i\pm \frac{1}{2}}$ includes an additional term, that takes into account the source of flux due to the Fokker-Planck term $\partial_x(x\ueps)$, see  \cite{CarrilloChertock2015} for further details. 
\end{rem}

\begin{rem}

A consideration on the relationship between $\Delta x$ and $\varepsilon$ has to be made. In fact we will use the scheme to check the order of convergence in estimate \eqref{eq:bessel_rate_eps}. In order to get numerically the correct order of convergence for $\varepsilon$ small, one should pick the size of the cells smaller or at least equal to the value of $\varepsilon$: $\Delta x \leq \varepsilon$. The intuition behind this requirement is sketched in Figure \ref{Fig:sizeEps}. The parameter $\varepsilon$ controls the level of nonlocal interaction in equation \eqref{eq:nonlocal_PDE_one_dimension} and as it decreases towards zero the non-locality diminishes its range, meaning that \eqref{eq:nonlocal_PDE_one_dimension} approaches the local problem \eqref{eq:local_PDE}. Therefore, for smaller values of $\varepsilon$, to ensure that we still detect the nonlocal interactions through the finite volume scheme \eqref{eqs:FVschem}, the size of the cells must have at least the same size of $\varepsilon$.

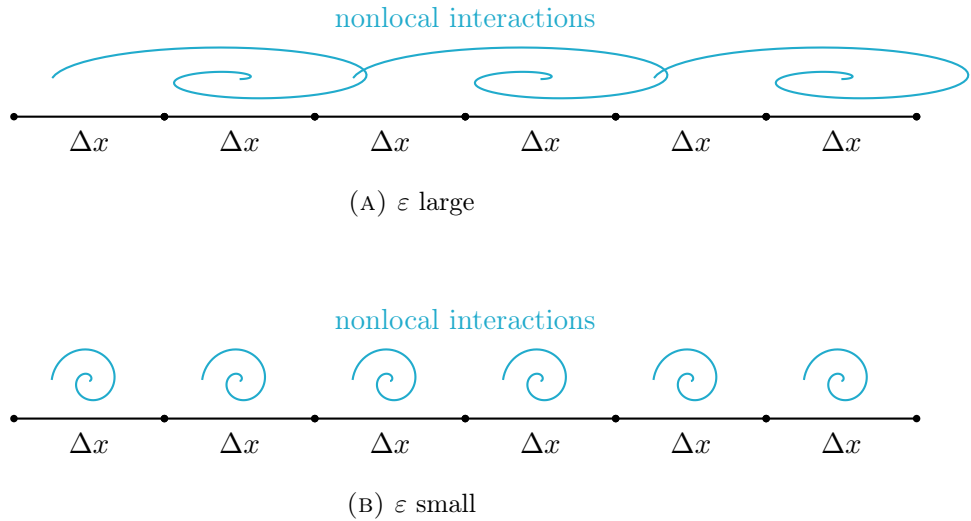
\begin{figure}[H]
\begin{subfigure}{0.7 \textwidth}
\begin{tikzpicture}
\def\lineLength{12}  %
\def\segments{6}     %
\def\spiralLoops{3}  %
\def\spiralHeight{1} %

\def\segmentLength{\lineLength/\segments}

\draw[thick] (0,0) -- (\lineLength, 0);

\node[above, color=ballblue, outer sep=2pt,fill=white] at (6, 1) {nonlocal interactions};

\foreach \i in {1,...,\segments} {
    \def\startX{\segmentLength*(\i-1/2)}
    \def\endX{\segmentLength*\i}

    \node[fill=black, circle, inner sep=1pt] at ({\startX-1}, 0) {}; %
    \node[fill=black, circle, inner sep=1pt] at ({\endX}, 0) {};   %
    \node[below, outer sep=2pt,fill=white] at ({\startX}, 0) {$\Delta x$};
}

\foreach \i in {2, 4, ...,\segments} {
    \def\startX{\segmentLength*(\i-1/2)}
    \def\endX{\segmentLength*\i}

    \node[fill=black, circle, inner sep=1pt] at ({\startX-1}, 0) {}; %
    \node[fill=black, circle, inner sep=1pt] at ({\endX}, 0) {};   %

    \draw[thick, color=ballblue, domain=0:0.5, samples=100, variable=\t] plot[parametric]
        ({\startX + 5 * \segmentLength * 0.5 * \t * cos(360 * \spiralLoops * \t)}, 
         {-0.5 + \spiralHeight + sin(360 * \spiralLoops * \t) * \segmentLength * 0.5 * \t});

}

\end{tikzpicture}
\caption{$\varepsilon$ large}
\end{subfigure}

\begin{subfigure}{0.7 \textwidth}
\begin{tikzpicture}
\def\lineLength{12}  %
\def\segments{6}     %
\def\spiralLoops{3}  %
\def\spiralHeight{1} %

\def\segmentLength{\lineLength/\segments}

\draw[thick] (0,-1) -- (\lineLength, -1);

\node[above, color=ballblue, outer sep=2pt,fill=white] at (6, 0) {nonlocal interactions};

\node[fill=white, circle, inner sep=1pt] at (0, 1.5) {};   %

\foreach \i in {1,...,\segments} {
    \def\startX{\segmentLength*(\i-1/2)}
    \def\endX{\segmentLength*\i}

    \node[fill=black, circle, inner sep=1pt] at ({\startX-1}, -1) {}; %
    \node[fill=black, circle, inner sep=1pt] at ({\endX}, -1) {};   %
    \node[below, outer sep=2pt,fill=white] at ({\startX}, -1) {$\Delta x$};

    \draw[thick, color=ballblue, domain=0:0.5, samples=100, variable=\t] plot[parametric]
        ({\startX + \segmentLength * 0.5 * \t * cos(360 * \spiralLoops * \t)}, 
         {-1.5 + \spiralHeight + sin(360 * \spiralLoops * \t) * \segmentLength * 0.5 * \t});
}
\end{tikzpicture}
\caption{$\varepsilon$ small}
\end{subfigure}
\caption{Finite Volume scheme: size of cells and $\varepsilon$}
\label{Fig:sizeEps}
\end{figure}
\end{rem}

\section{Numerical experiments}
We use the scheme \eqref{eqs:FVschem} to analyse the estimate of Theorem \ref{thm:bessel_rate_eps}. The scheme is implemented on one-dimensional domains with varying initial and boundary conditions, allowing us to investigate various aspects of convergence as well as the properties of the equation \eqref{eq:nonlocal_PDE_one_dimension}. In the following, we will denote by $u_0$ the solution $u$ of the local models \eqref{eq:local_PDE} and \eqref{eq:local_PDE_FoPla}.

\subsection{Sharpness of the error estimate between the local and the nonlocal problem}

In Figure \ref{Fig:GaussConv}  we report for different times the convergence results in Wasserstein distance $\mathcal{W}_2$ between the numerical approximations of $u_0$ and $\ueps$, solutions of \eqref{eq:local_PDE} and \eqref{eq:nonlocal_PDE_one_dimension} respectively, using the finite volume scheme \eqref{eqs:FVschem}. 
The simulations in Figure \ref{Fig:GaussConv}  for different values of $\varepsilon$ start all with a Gaussian initial datum $u^0 = \frac{1}{\sqrt{2\pi}} e^{-|x|^2/2}$. These experiments should resemble the behaviour of the equations on the whole of $\mathbb{R}$. The convergence plot validates the result of Theorem \ref{thm:bessel_rate_eps}. In particular, we first observe that the estimate, provided by \eqref{eq:bessel_rate_eps}, is global in $\varepsilon$, meaning its proof imposes no restrictions on the values of $\varepsilon$. Figure \ref{Fig:GaussConv} shows how the rate of convergence for $\mathcal{W}_2(\ueps, u_0)$, with this specific initial datum, is linear in $\varepsilon$ for small values of $\varepsilon$, while for larger values of $\varepsilon$ there is a transition and the convergence deteriorates behaving more like $\varepsilon^{\beta}$ for $\beta < 1/4$. We have clearly that $
    \varepsilon < \sqrt{\varepsilon} < \varepsilon^{\frac{1}{4}}
$
for $\varepsilon < 1$, and $\varepsilon^{\beta} < \varepsilon^{\frac{1}{4}} < \sqrt{\varepsilon} < \varepsilon$ for $\varepsilon > 1$. Therefore, the numerical results provided by our finite volume scheme demonstrate that the estimate of \eqref{eq:bessel_rate_eps} is in fact the best, \textit{global} in $\varepsilon$ estimate, while the rate of $\eps$ from \cite{art_new_formula_Wasserstein} is valid for small values of $\varepsilon$.

\subsection{Periodic boundary conditions}

In Figure \ref{Fig:ParabConvPBC}, we test the case of periodic boundary conditions, starting from a compactly supported initial datum, $u^0 = (1-|x|^2)_+$. The experiment resembles the behaviour that equation \eqref{eq:nonlocal_PDE_one_dimension} has on the whole of $\mathbb{R}$ while the support does not touch the boundary of the domain. Once the support reaches the boundary we see the expected behavior converging towards the constant solution. We observe the expected rate of convergence for $\mathcal{W}_2(\ueps, u_0)$, similar to the experiments of Figure \ref{Fig:GaussConv} during the whole evolution despite the periodic boundary conditions. The order of convergence remains indeed linear in $\varepsilon$ for small values of $\varepsilon$ for all times of the simulation. These numerical experiments corroborate again the results obtained in \cite{art_new_formula_Wasserstein}.

\begin{figure}[H]
\centering
\begin{subfigure}{0.75 \textwidth}
\includegraphics[width = \textwidth]{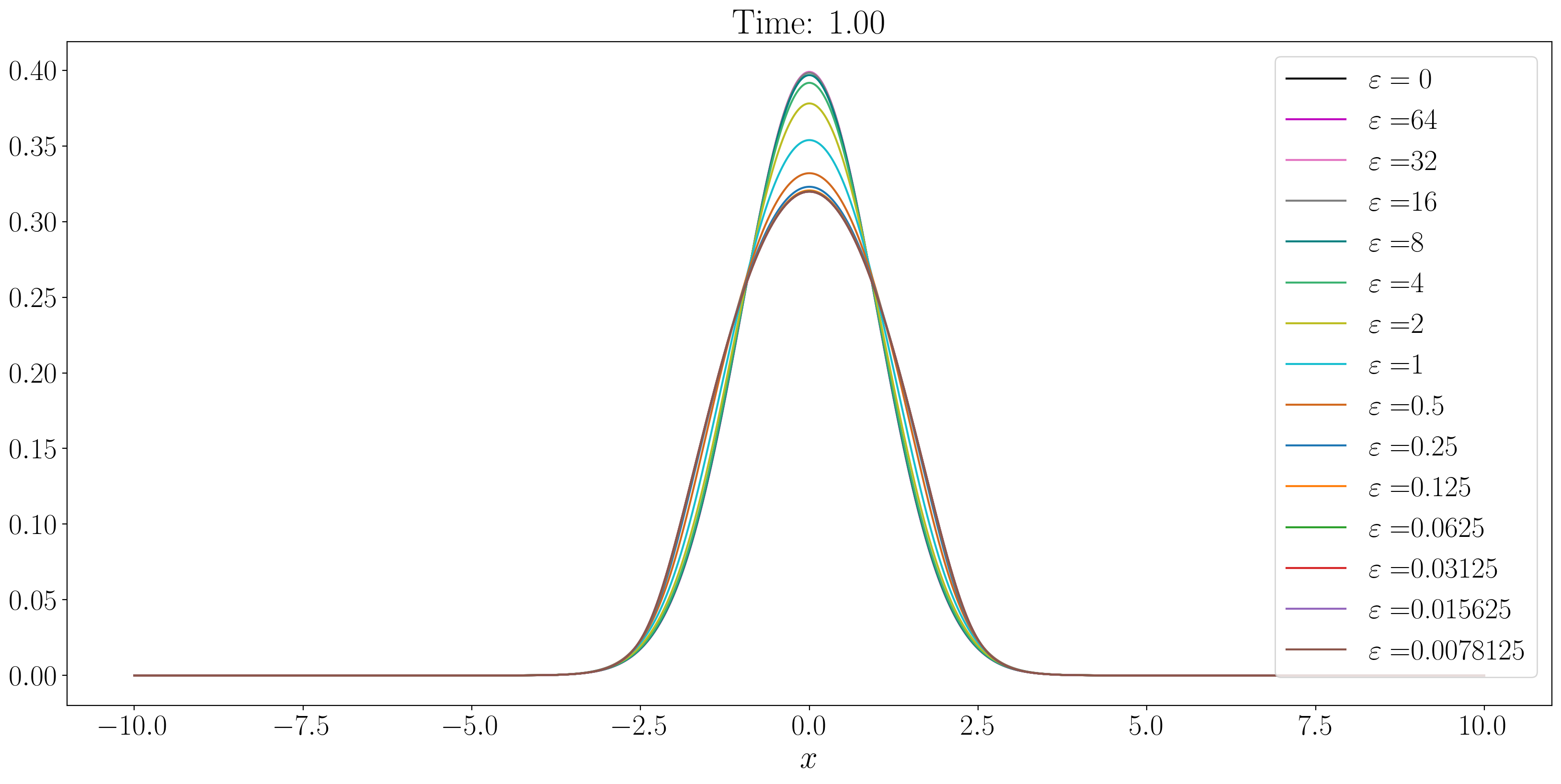} 
\end{subfigure}
\hfill
\centering
\begin{subfigure}{0.75 \textwidth}
\centering
\includegraphics[width = \textwidth]{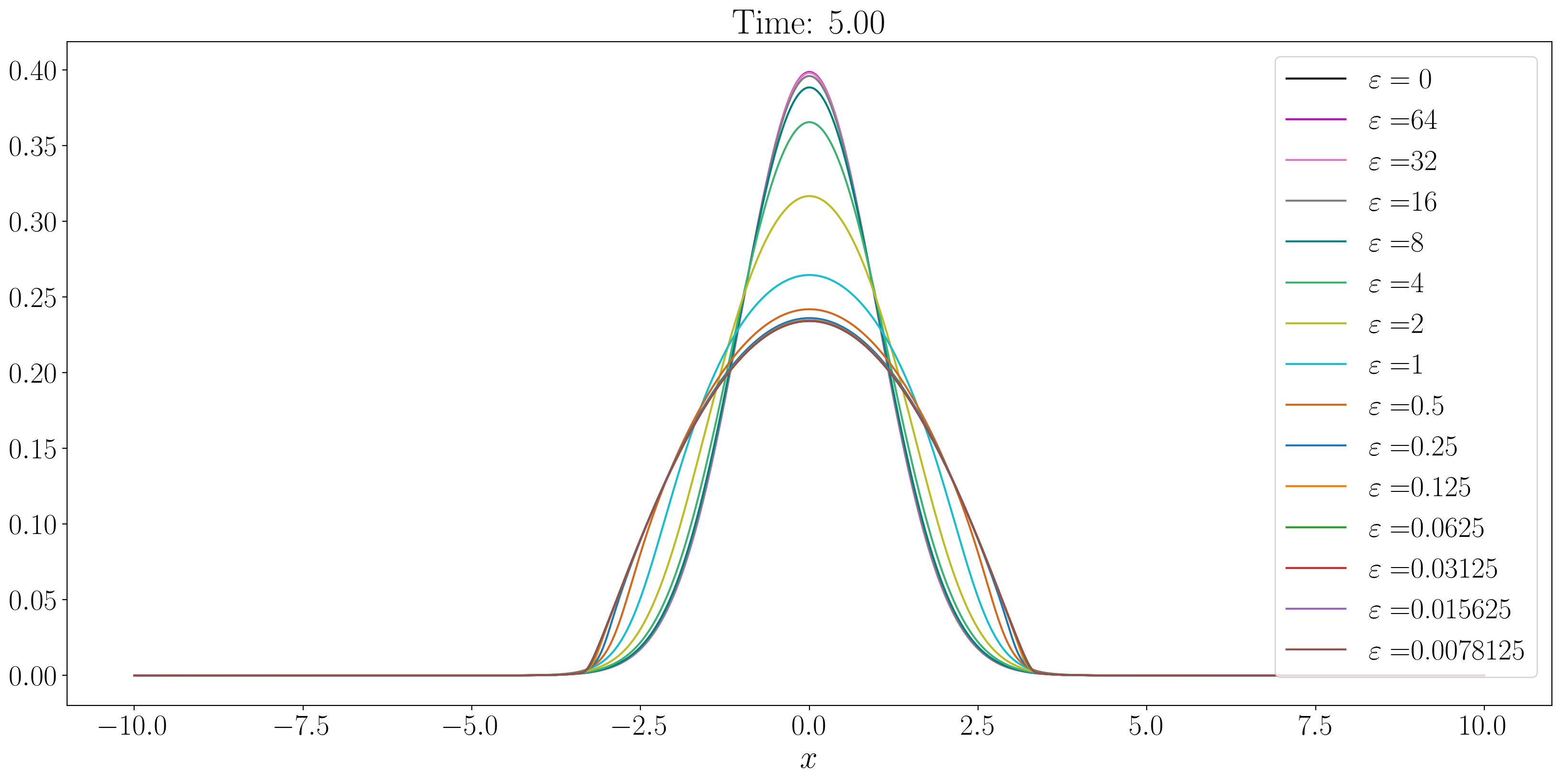} 
\end{subfigure}
\hfill
\centering
\begin{subfigure}{0.75 \textwidth}
\centering
\includegraphics[width =\textwidth]{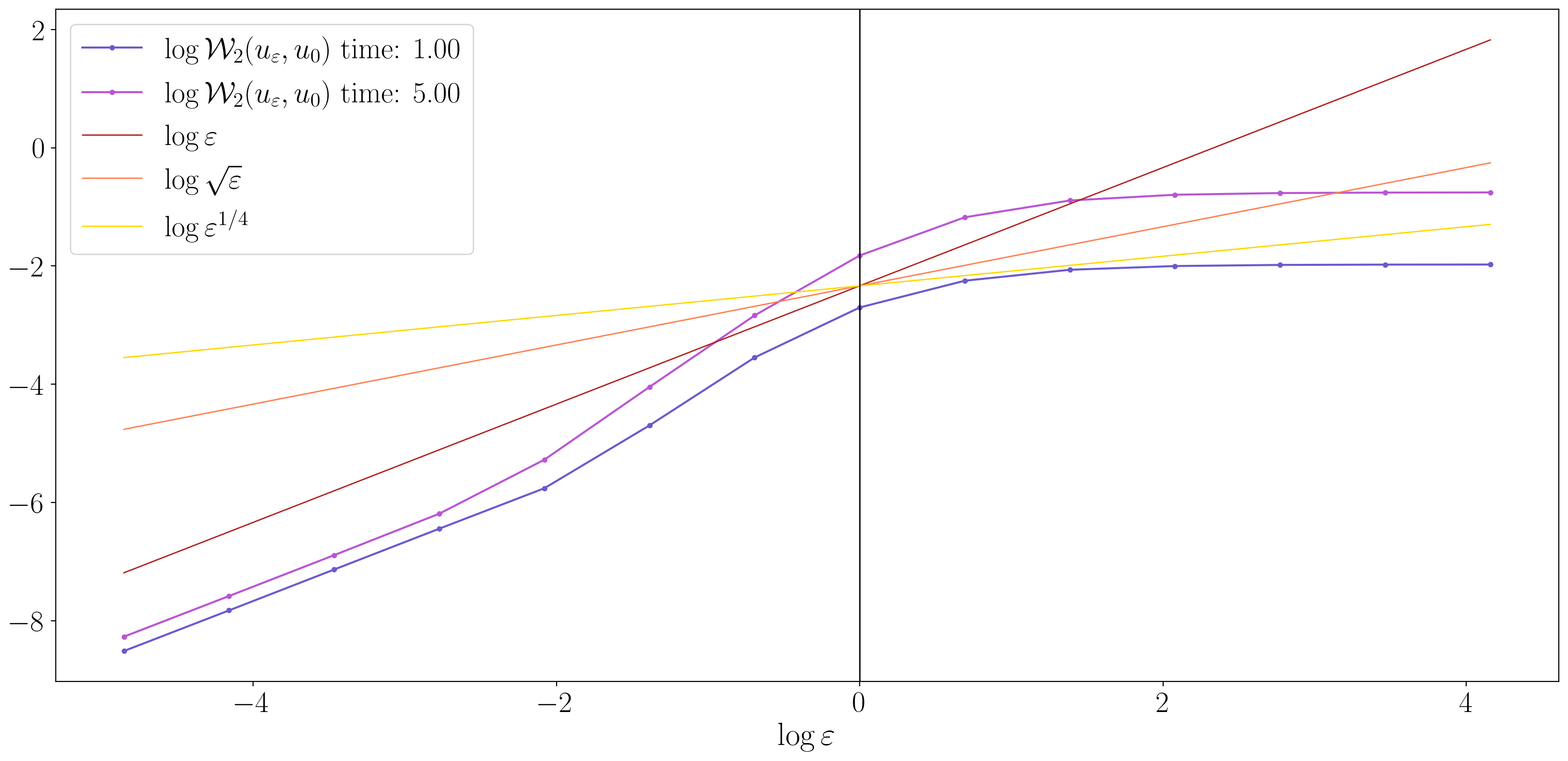} 
\end{subfigure}
\caption{Porous medium equation versus blob method: order of convergence for $\mathcal{W}_2(\ueps, u_0)$, $u_0$ solution of \eqref{eq:local_PDE} and $\ueps$ solution of \eqref{eq:nonlocal_PDE_one_dimension}. The computational domain is $\Omega = (-10, 10)$,  $u^0(x) = (1/\sqrt{2\pi}) \exp(-|x|^2/2)$ is the initial datum, a uniform grid of $N=2^{12}$ cells was used with  $\Delta t = 0.01$.}
\label{Fig:GaussConv}
\end{figure}

\begin{figure}[H]
\centering
\begin{subfigure}{0.75 \textwidth}
\includegraphics[width = \textwidth]{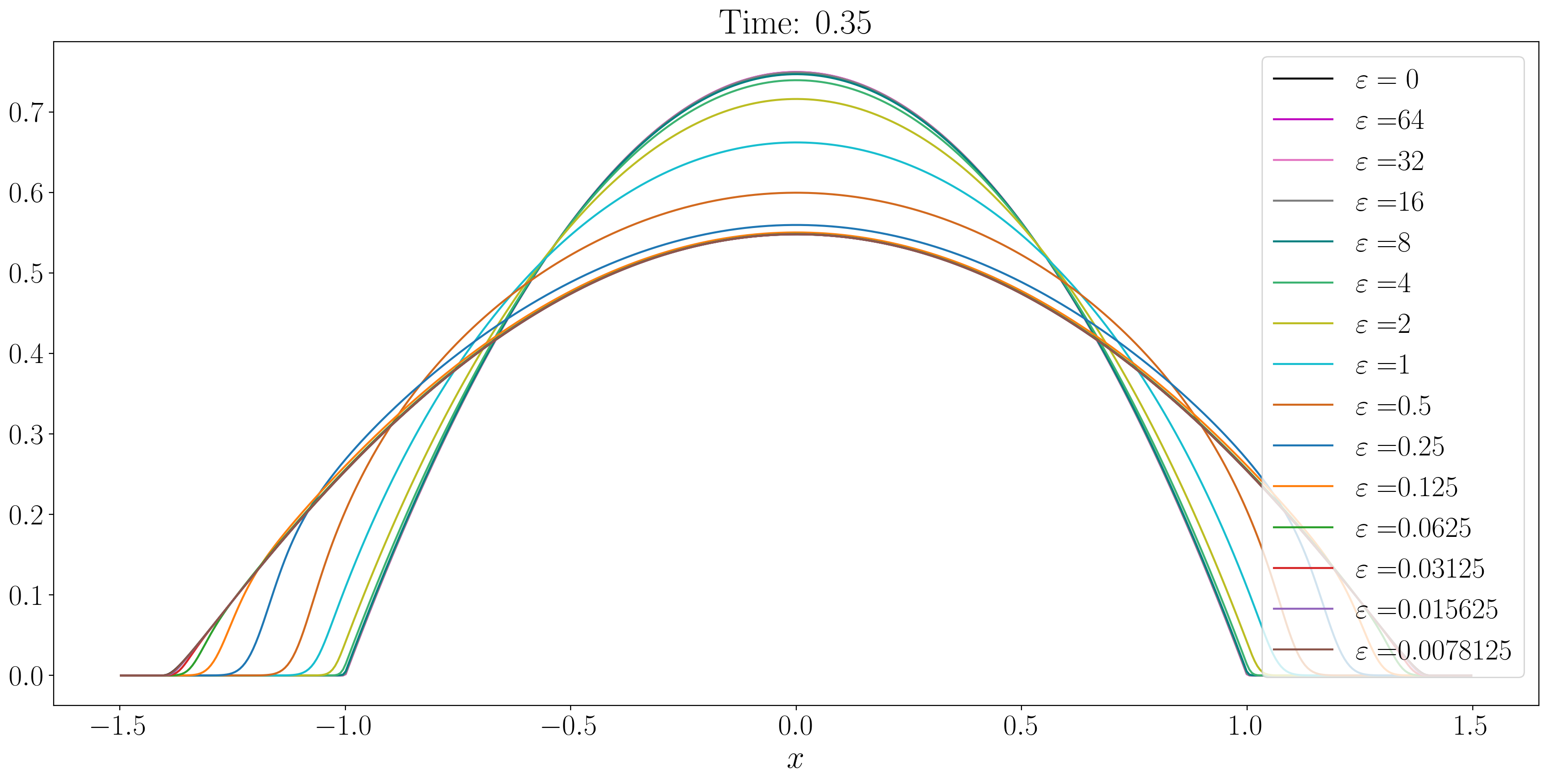} 
\end{subfigure}
\hfill
\centering
\begin{subfigure}{0.75 \textwidth}
\centering
\includegraphics[width = \textwidth]{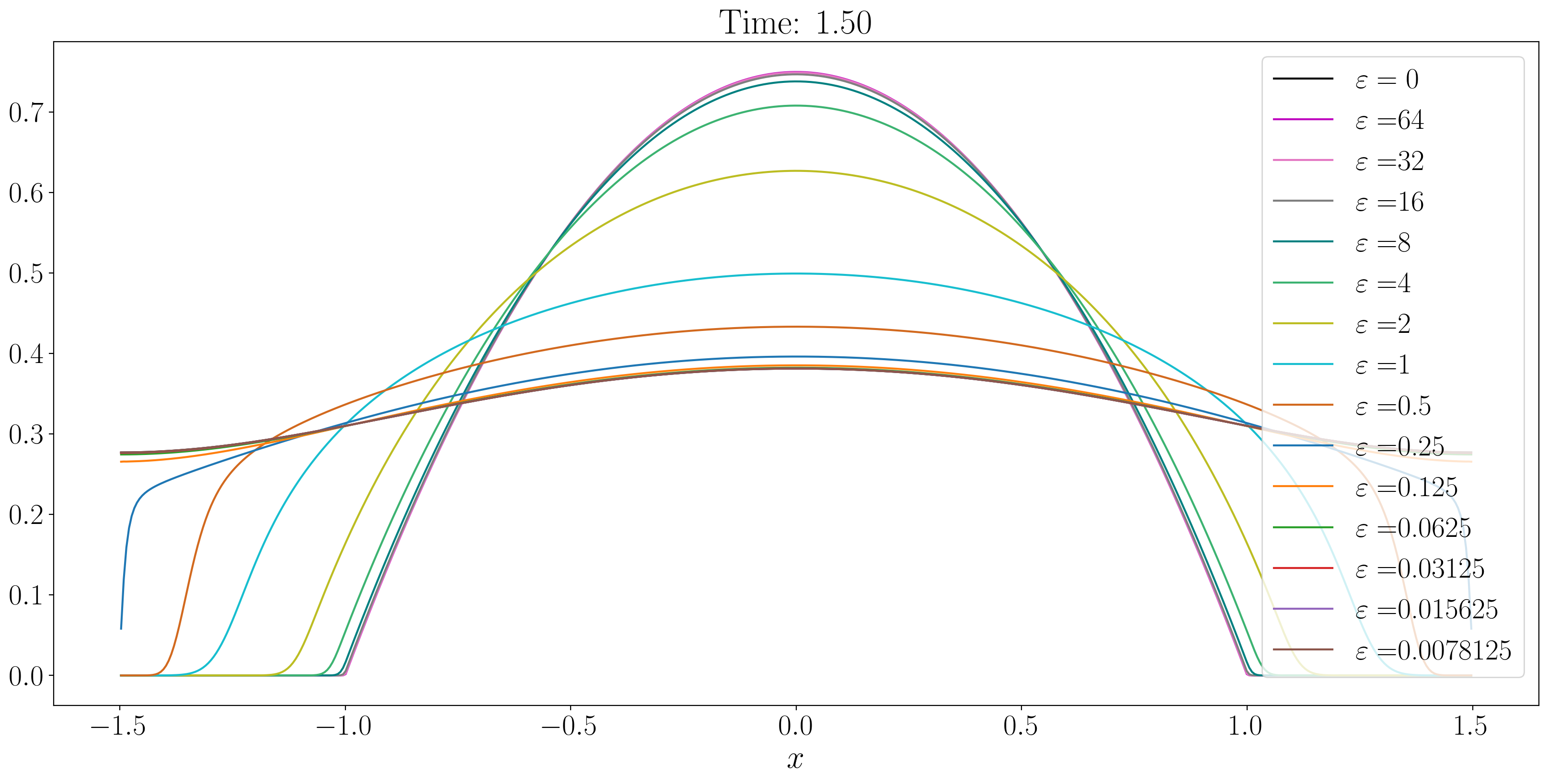} 
\end{subfigure}
\hfill
\centering
\begin{subfigure}{0.75 \textwidth}
\centering
\includegraphics[width =\textwidth]{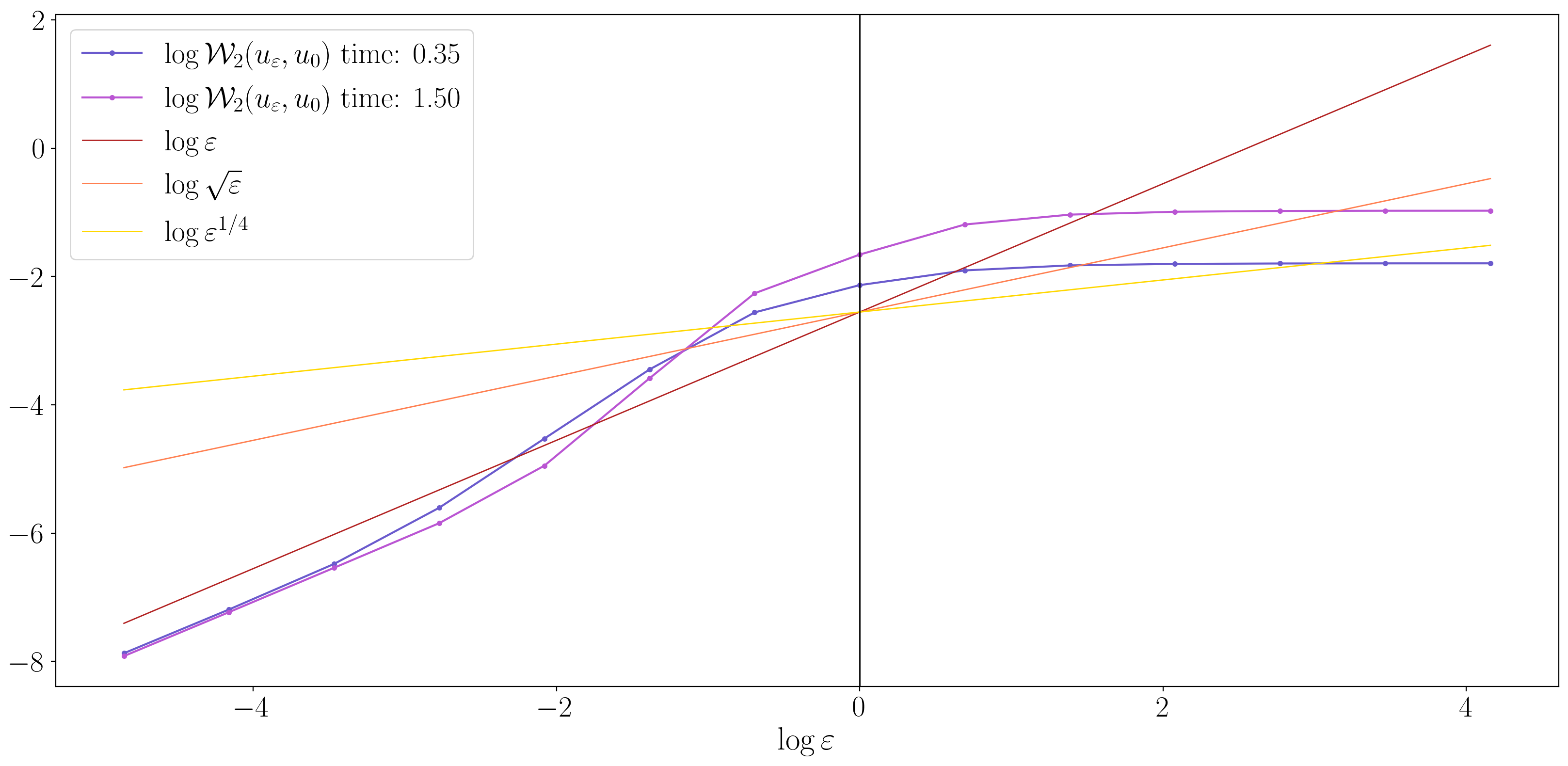} 
\end{subfigure}
\caption{Porous medium equation and blob method: order of convergence for $\mathcal{W}_2(\ueps, u_0)$, $u_0$ solution of \eqref{eq:local_PDE} and $\ueps$ solution of \eqref{eq:nonlocal_PDE_one_dimension}. The computational domain $\Omega = (-3, 3)$,  $u^0(x) = (1-x^2)_+$ is the initial datum, a uniform grid of $N=2^{10}$ cells with  $\Delta t = 0.01$ was used, periodic boundary conditions were imposed.}
\label{Fig:ParabConvPBC}
\end{figure}

\subsection{No flux boundary conditions}
In Figure \ref{Fig:ParabConvNoFlux} we test on a bounded domain the influence of the boundary on the convergence rate for $\mathcal{W}_2(\ueps, u_0)$, where $u_0$ and $\ueps$ solve respectively \eqref{eq:local_PDE} and \eqref{eq:nonlocal_PDE_one_dimension}. Starting from the same compactly supported initial datum, that we used in the experiments in Figure \ref{Fig:ParabConvPBC}, we observe how the order of convergence for $\varepsilon$ small is clearly linear in $\varepsilon$ initially, as Figure \ref{Fig:ParabConvNoFlux} displays. However, after the solutions touch the boundary, and the density starts to accumulate at the boundary, one can observe a decrease of the order of convergence, that becomes approximately proportional to $\sqrt{\varepsilon}$, as reported in Figure \ref{Fig:ParabConvNoFlux}. This experiment empirically highlights how the order of convergence may vary with respect to the case of the whole space, precisely due to the presence of the no-flux boundary conditions in contrast to \cite{art_new_formula_Wasserstein}, and how the interaction with the boundary might generally slow down convergence. 
In  addition, we can notice how, imposing no-flux boundary conditions, lead to the accumulation of the density at the boundary points because the repulsive kernel pushes the density to concentrate at the boundary. The final steady states for $\ueps$  get closer to the constant function as $\varepsilon\to 0$, except for the two boundary spikes, as one can observe in Figure \ref{Fig:ParabConvNoFlux} for intermediate times. 

\subsection{Fokker-Planck problem and the effect of a bounded domain}

In Figure \ref{Fig:UnifConv} we check the convergence for the analogous Fokker-Planck problem \eqref{eq:nonlocal_PDE_one_dimension_FoPla}. All the simulations start in this case from a uniform distribution $u^0 = \frac{1}{|\Omega|} + 10$. Despite the support of the initial datum being the whole bounded domain $\Omega$, the evolution converges towards the expected stationary state of the Fokker-Planck equation on the whole of $\mathbb{R}$ emulating the behavior on the whole line.  We can still observe how the convergence follows a rate proportional $\varepsilon^{\beta}$ with $\beta<1/4$ for $\varepsilon \gg 1$, while, in particular for small times, the convergence for $\varepsilon \ll 1$ follows a rate of $\sqrt{\varepsilon}$, with a transition, where it becomes almost linear for $\varepsilon$ close to 1. For larger times the behaviour of $\mathcal{W}_2(\ueps, u_0 )$ seems to improve in terms of convergence, keeping however $\sqrt{\varepsilon}$ as convergence rate for the smallest values. This experiment not only validates the estimate \eqref{eq:bessel_rate_eps}, but it also show how the estimate is actually sharp for some initial data. 
The interaction with the boundary and the boundedness of the domain turns then to be decisive for obtaining the rate $\sqrt{\varepsilon}$, as in the previous subsection.

\begin{figure}[H]
\centering
\begin{subfigure}{0.75 \textwidth}
\includegraphics[width = \textwidth]{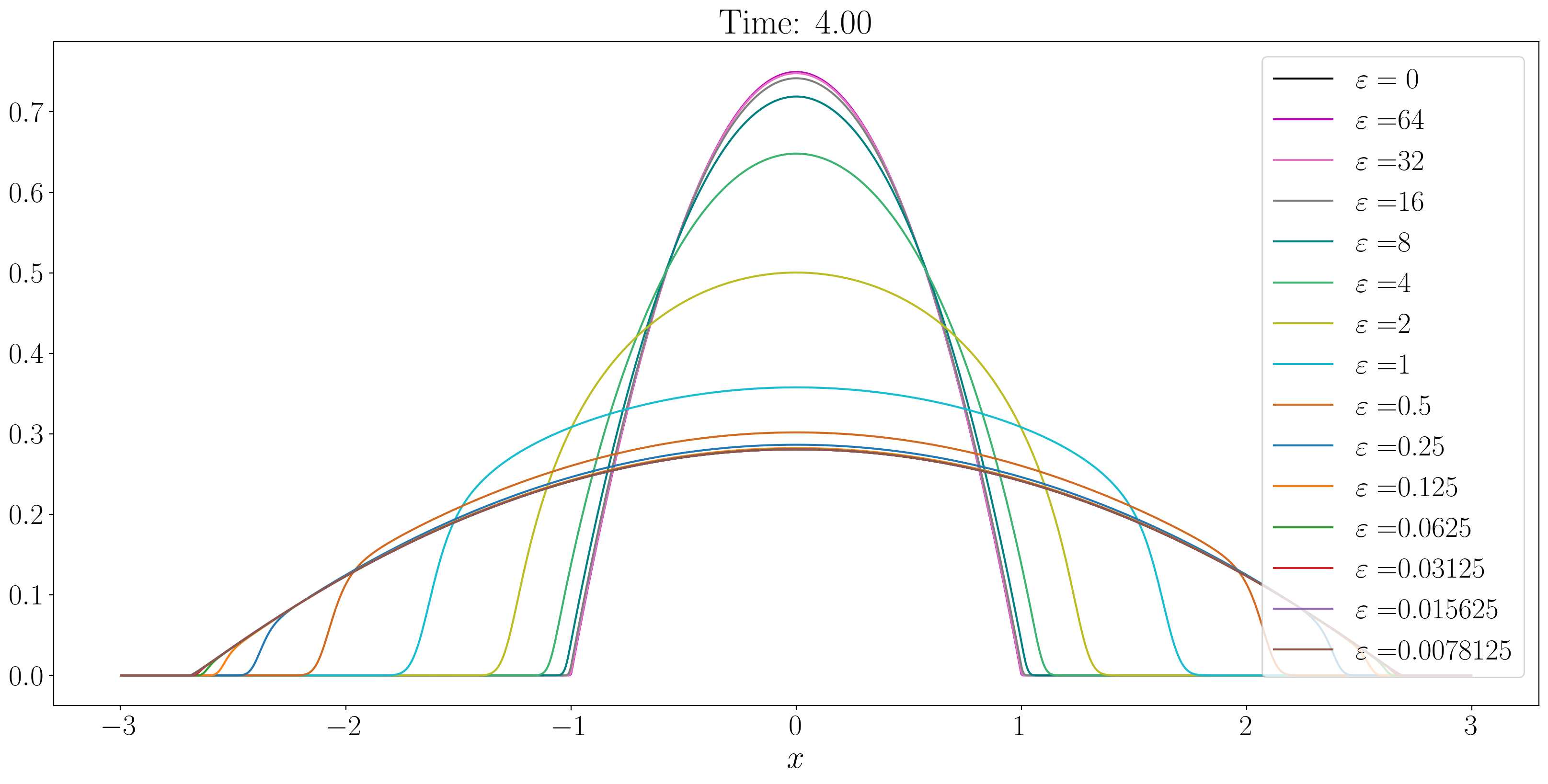} 
\end{subfigure}
\hfill
\centering
\begin{subfigure}{0.75 \textwidth}
\centering
\includegraphics[width = \textwidth]{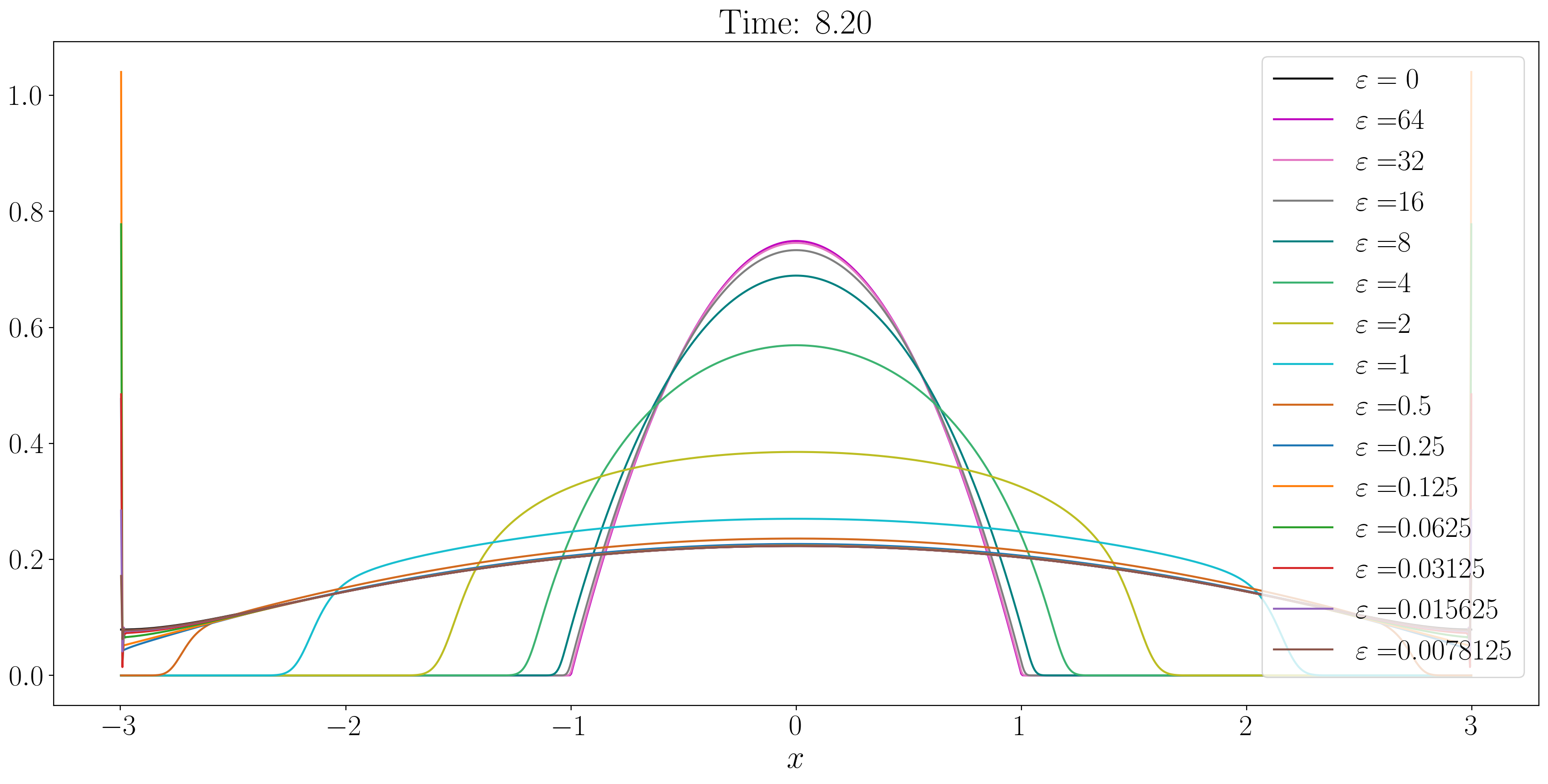} 
\end{subfigure}
\hfill
\centering
\begin{subfigure}{0.75 \textwidth}
\centering
\includegraphics[width =\textwidth]{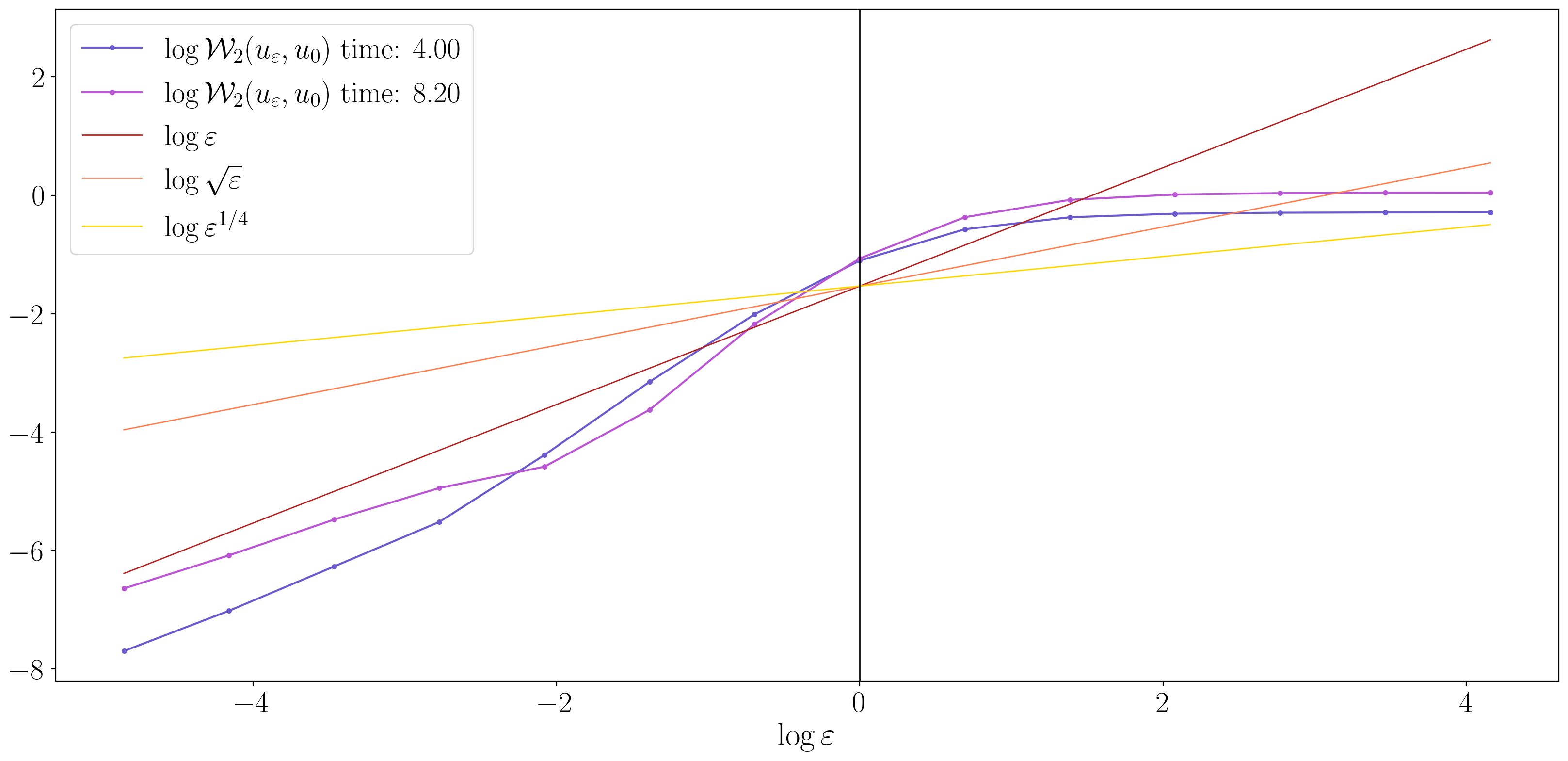} 
\end{subfigure}
\caption{Porous medium equation and blob method: order of convergence for $\mathcal{W}_2(\ueps, u_0)$, $u_0$ solution of \eqref{eq:local_PDE} and $\ueps$ solution of \eqref{eq:nonlocal_PDE_one_dimension}. The computational domain is $\Omega = (-3, 3)$,  $u^0(x) = (1-x^2)_+$ is the initial datum, a uniform grid of $N=2^{10}$ cells was used with $\Delta t = 0.01$, no-flux boundary conditions were imposed.}
\label{Fig:ParabConvNoFlux}
\end{figure}

\begin{figure}[H]
\centering
\begin{subfigure}{0.75 \textwidth}
\includegraphics[width = \textwidth]{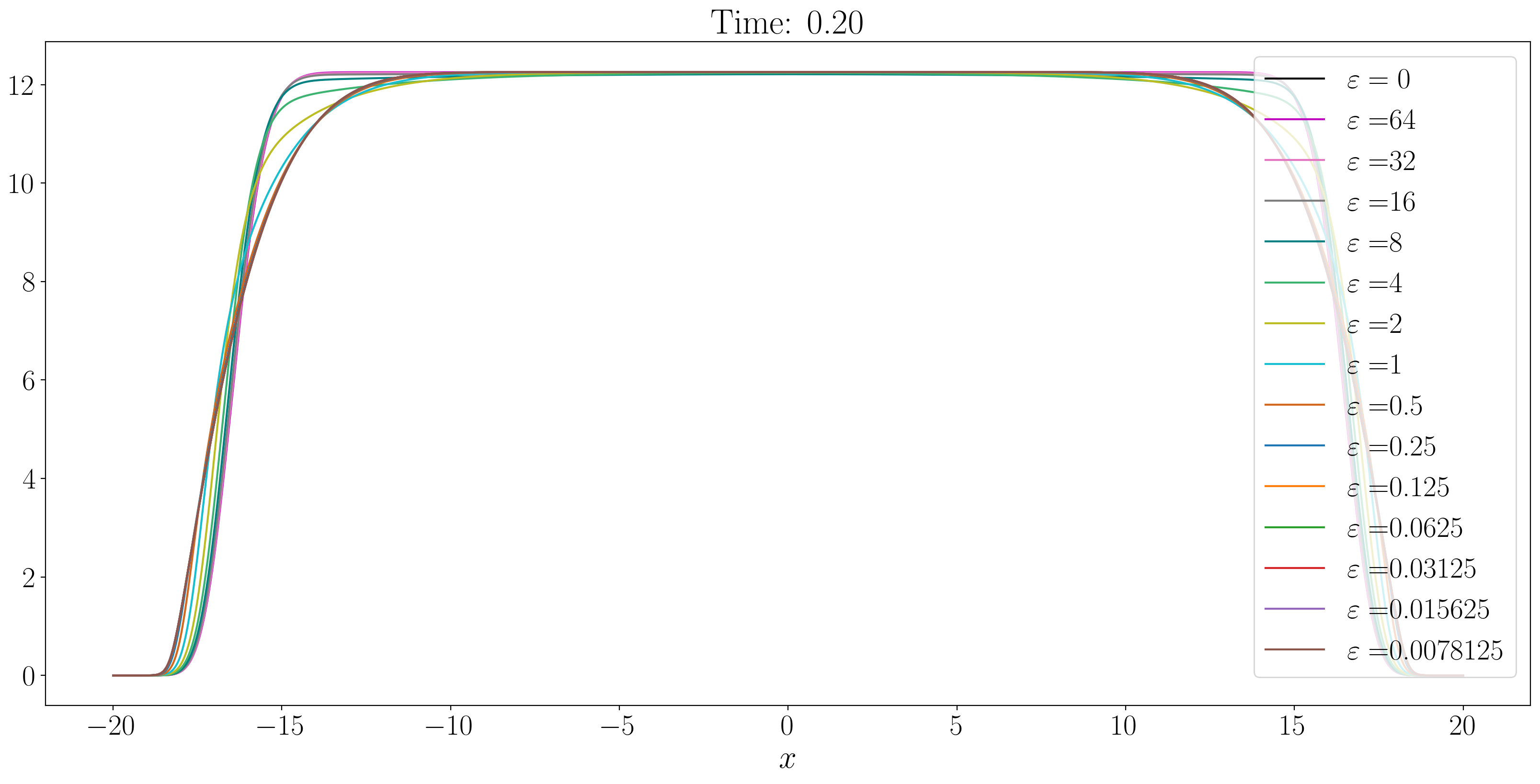} 
\end{subfigure}
\hfill
\centering
\begin{subfigure}{0.75 \textwidth}
\centering
\includegraphics[width = \textwidth]{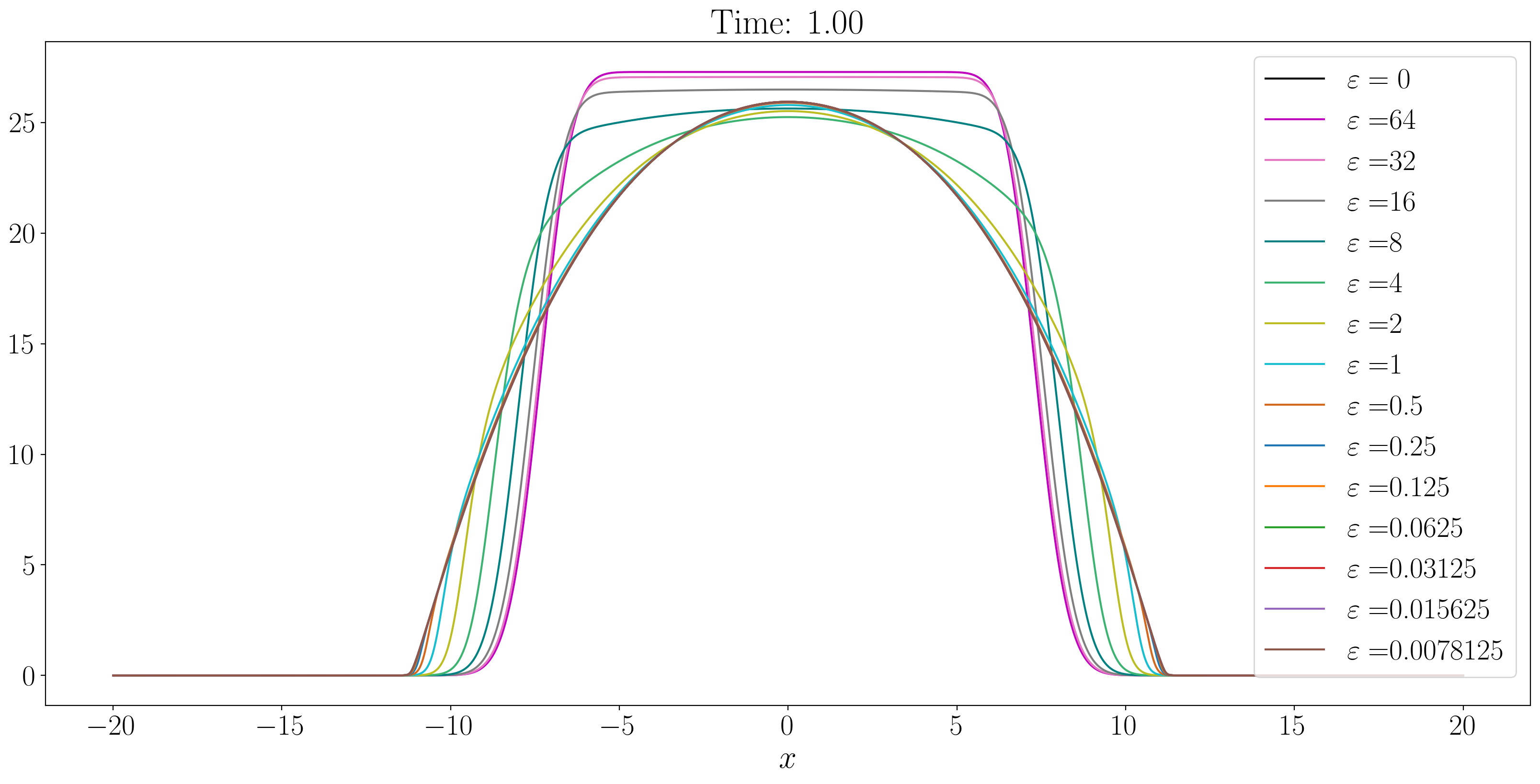} 
\end{subfigure}
\hfill
\centering
\begin{subfigure}{0.75 \textwidth}
\centering
\includegraphics[width =\textwidth]{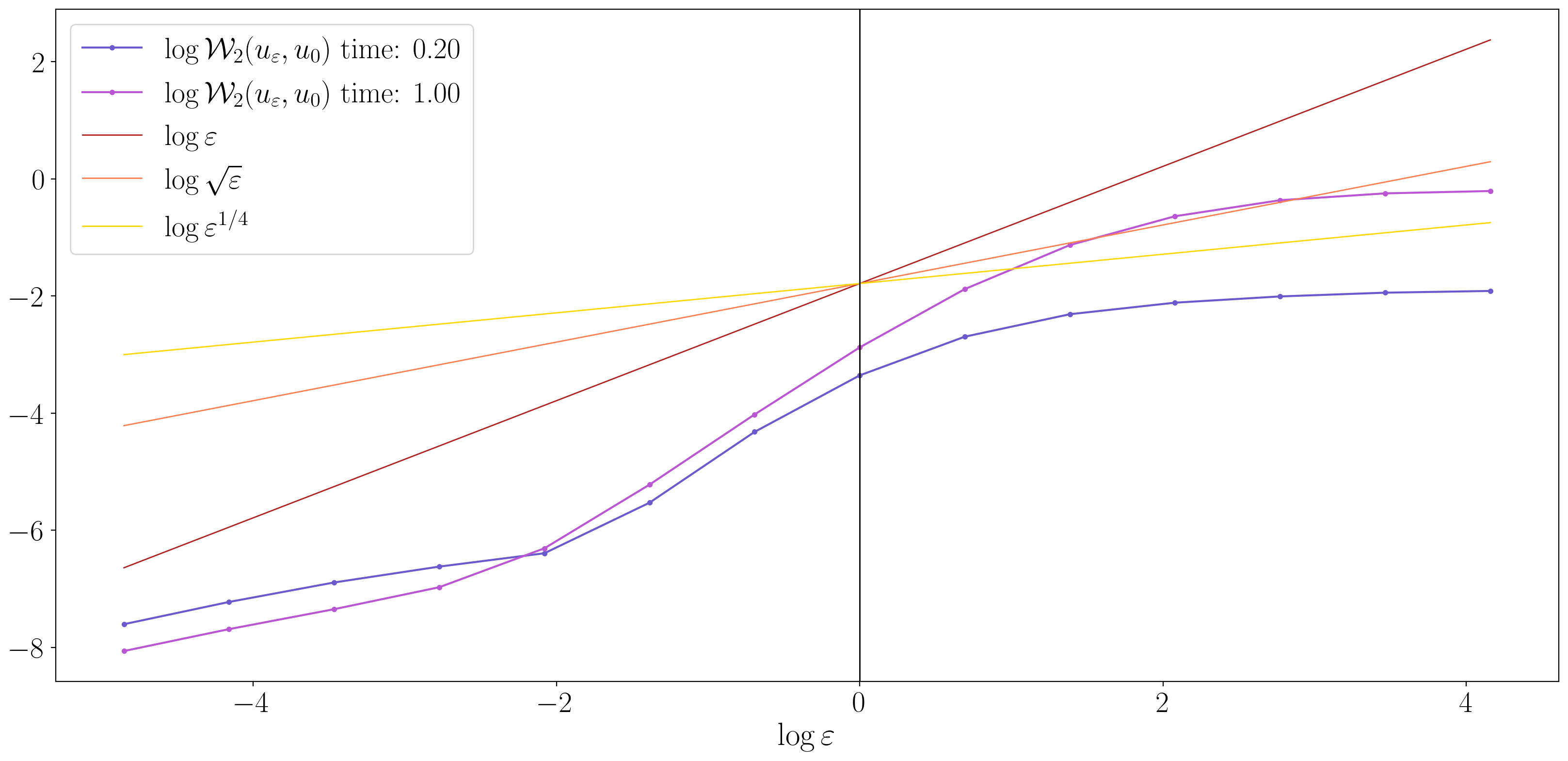} 
\end{subfigure}
\caption{Porous medium Fokker-Planck equation and blob method: order of convergence for $\mathcal{W}_2(\ueps, u_0)$, $u_0$ solution of \eqref{eq:local_PDE_FoPla} and $\ueps$ solution of \eqref{eq:nonlocal_PDE_one_dimension_FoPla}. The computational domain is $\Omega = (-20, 20)$,  $u^0(x) = 1/|\Omega| + 10$ is the initial datum, a uniform grid of $N=2^{13}$ cells was used with $\Delta t = 0.01$.}
\label{Fig:UnifConv}
\end{figure}

\appendix
\section{Complementary material}

\subsection{Nonlocal and local models: acquisition of regularity}

In Figure \ref{Fig:RandomConv} we have the results for an interesting case, in which we consider an initial datum $u^0$, that is randomly generated and piecewise constant, with different values in each cell of the domain $\Omega$. We can see in Figure \ref{Fig:RandomConv} the same convergence behaviour previously observed: the Wasserstein distance $\mathcal{W}_2$ between $u_0$ and $\ueps$ initially decreases like $\sqrt{\varepsilon}$ for small values of $\varepsilon$, with convergence slowing down for $\varepsilon \gg 1$. As time increases, the convergence becomes faster for $\varepsilon \ll 1$, as reported in Figure \ref{Fig:RandomConv}. 

In addition in Figure \ref{Fig:RandomProfiles} we can test a property that holds true for the porous medium equation: starting from an initial datum $u^0(x) = u(x, 0) \in L^1(\mathbb{R}) \cap L^{\infty}(\mathbb{R})$, the solution $u_0(x,t)$ becomes immediately H\"{o}lder regular for $t>0$ (the reader can refer to \cite{ref:Urbano2008, ref:Santambrogio2015} for more information on this property). This is not the case for the solution $\ueps(x,t)$ of the nonlocal approximation: while the solution remains in 
 $L^1(\mathbb{R}) \cap L^{\infty}(\mathbb{R})$, it does not acquire the same regularity as 
$u_0$ as time progresses. This fact could be noticed also in Figure \ref{Fig:RandomConv}, where we can observe that $\ueps$ does not gain regularity for larger values of $\ueps$ in the center of the domain, where the density is driven by the Fokker-Planck term.

\begin{figure}[H]
\centering
\begin{subfigure}{0.75 \textwidth}
\includegraphics[width = \textwidth]{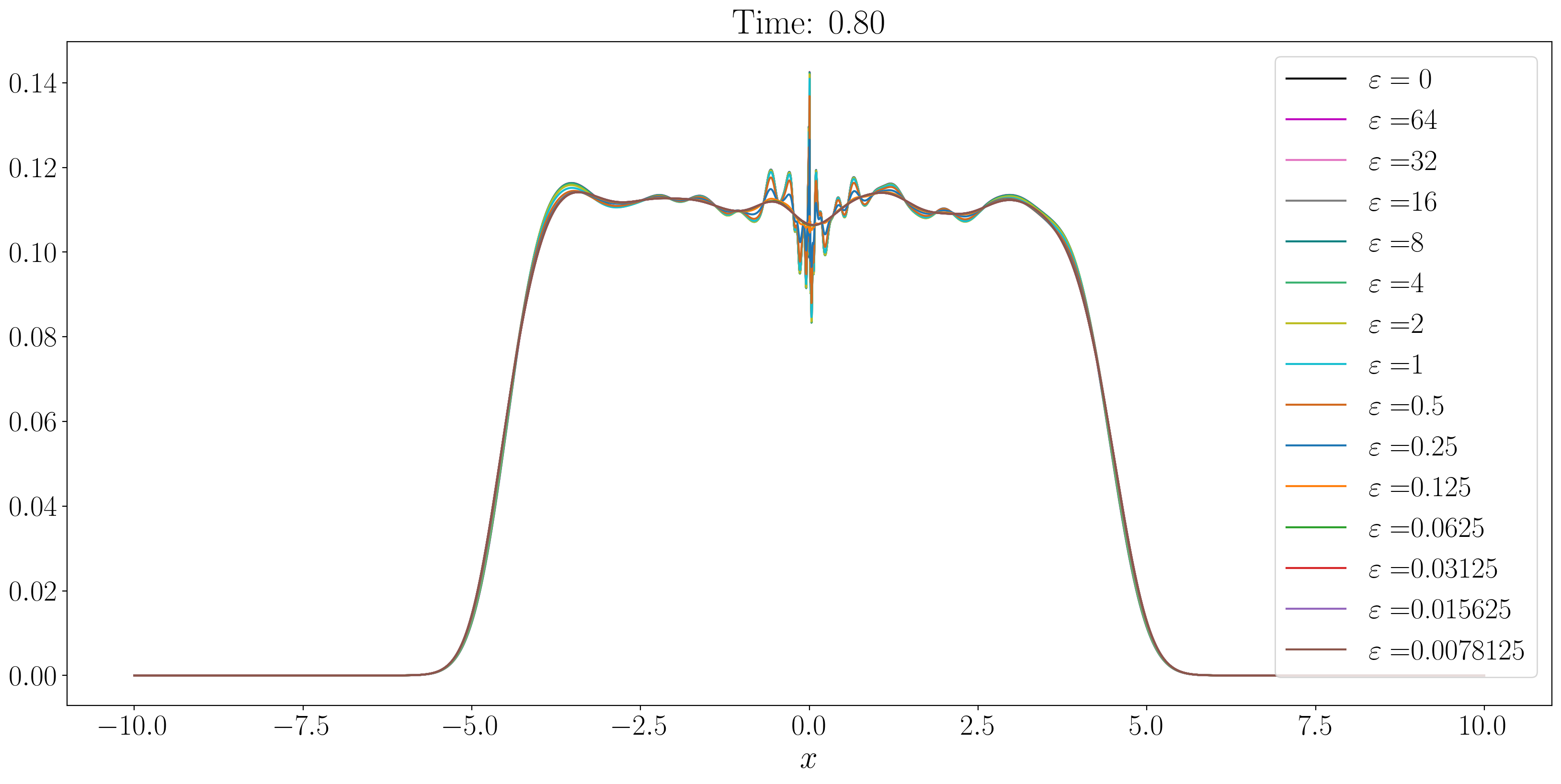} 
\end{subfigure}
\hfill
\centering
\begin{subfigure}{0.75 \textwidth}
\centering
\includegraphics[width = \textwidth]{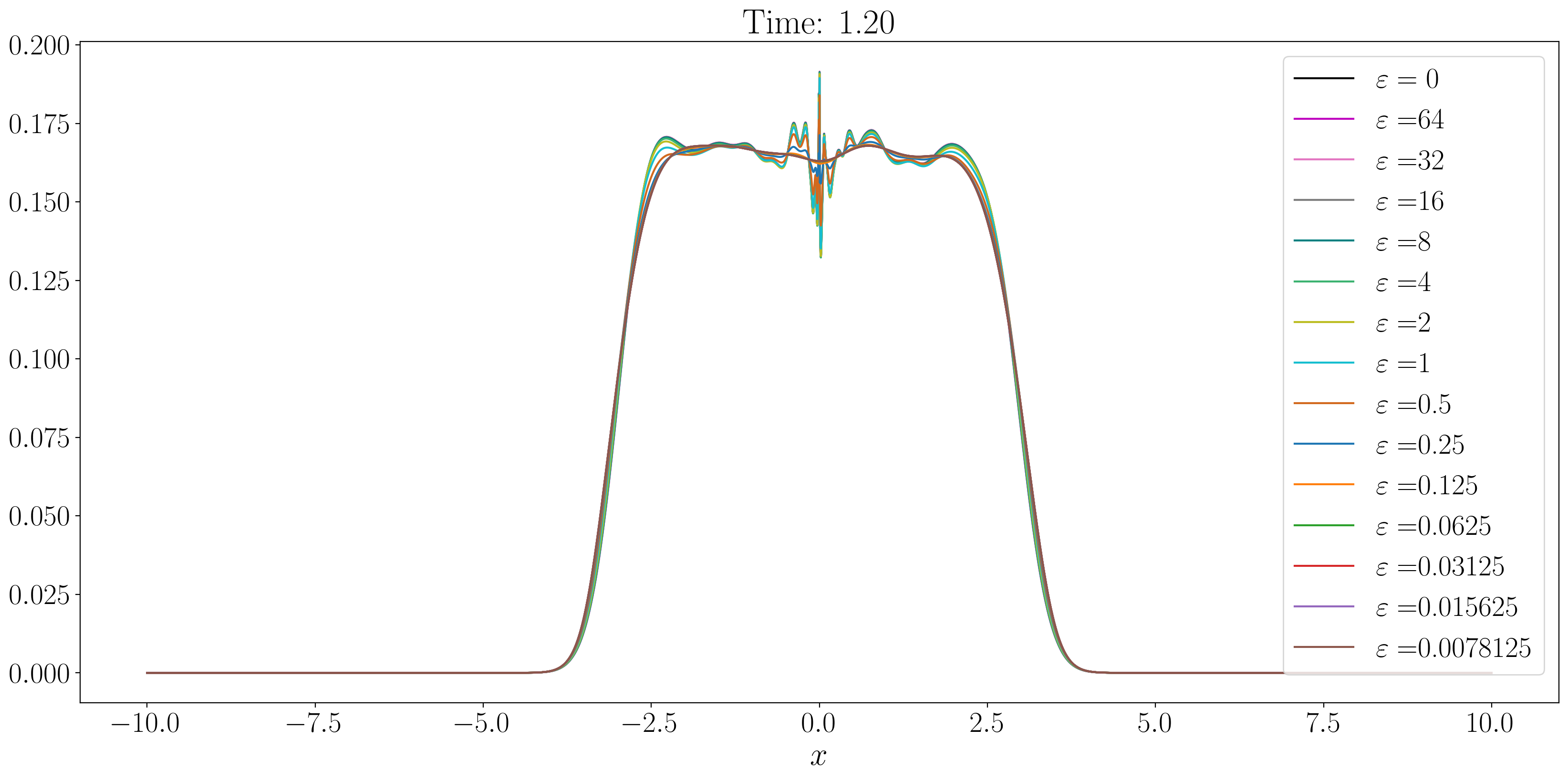} 
\end{subfigure}
\hfill
\centering
\begin{subfigure}{0.75 \textwidth}
\centering
\includegraphics[width =\textwidth]{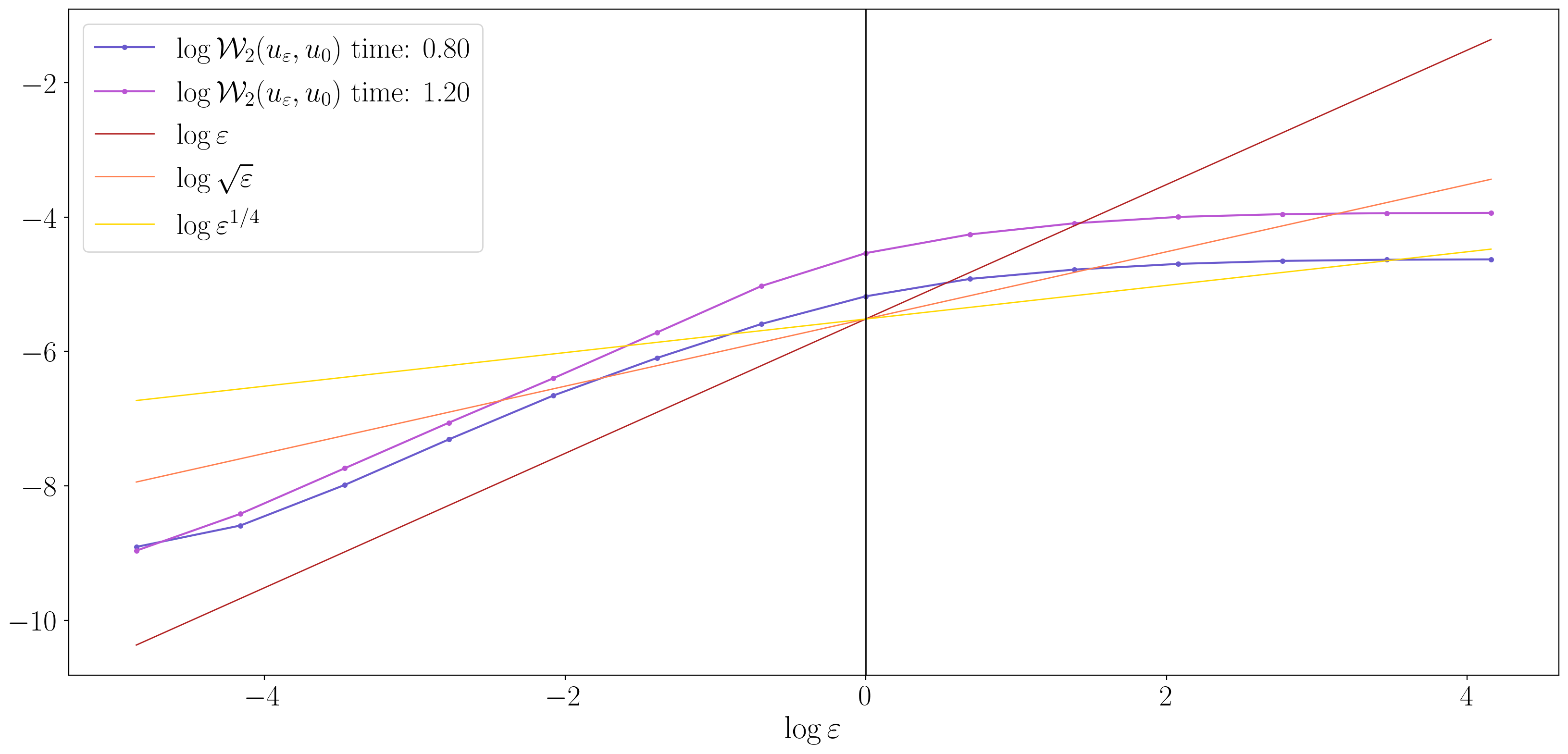} 
\end{subfigure}
\caption{Porous medium equation and blob method: order of convergence for $\mathcal{W}_2(\ueps, u_0)$, $u_0$ (solution of \eqref{eq:local_PDE_FoPla}) and $\ueps$ (solution of \eqref{eq:nonlocal_PDE_one_dimension_FoPla}). The computational domain is $\Omega = (-10, 10)$,  $u^0(x) = u(x, 0)$, the initial datum, is generated using random values on a uniform grid of $N=2^{12}$ cells, with $\Delta t = 0.01$.}
\label{Fig:RandomConv}
\end{figure}

\begin{figure}[H]
\centering
\begin{subfigure}{0.7 \textwidth}
\includegraphics[width = \textwidth]{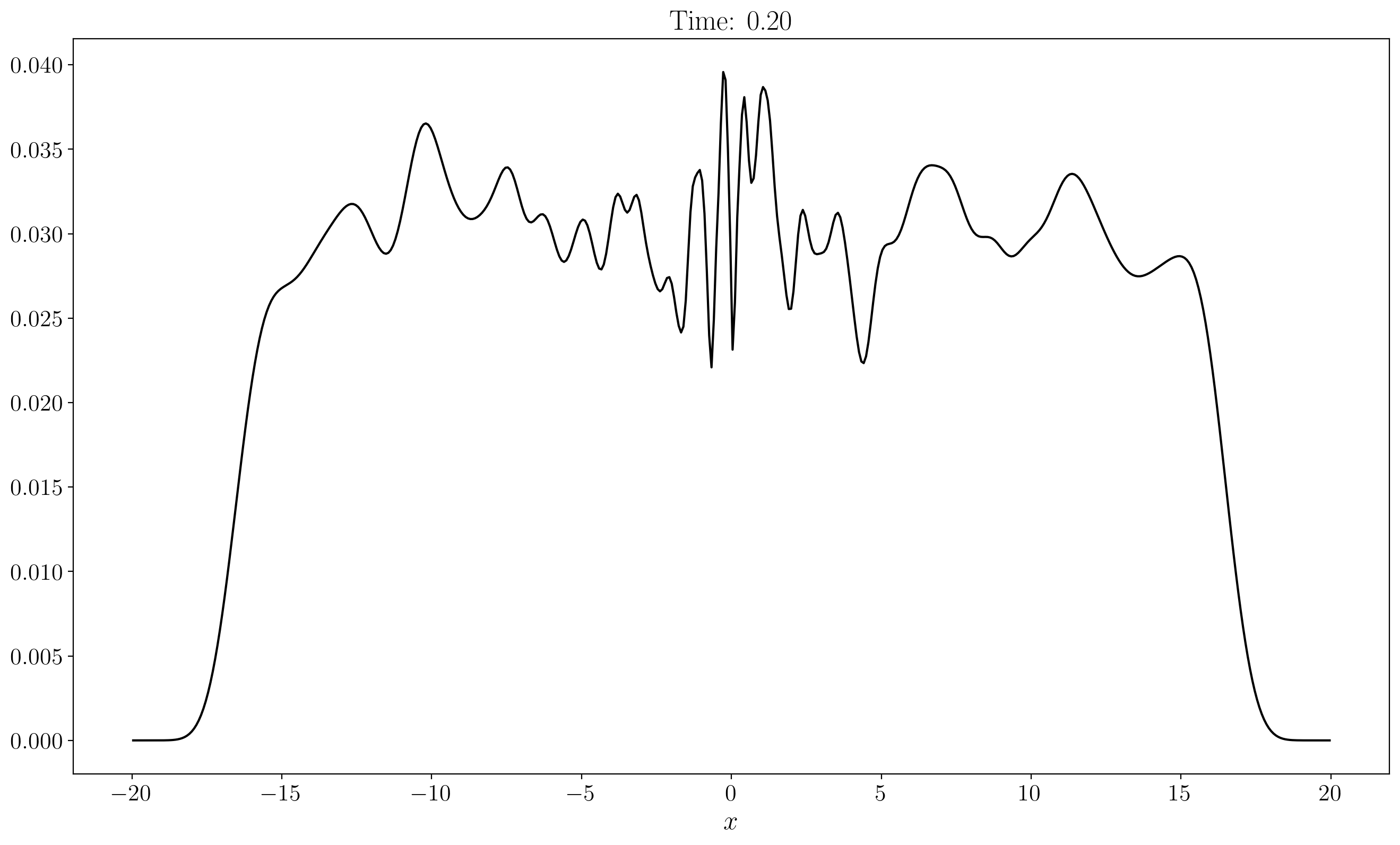} 
\caption{$\varepsilon = 0$}
\end{subfigure}
\hfill
\centering
\begin{subfigure}{0.7 \textwidth}
\centering
\includegraphics[width = \textwidth]{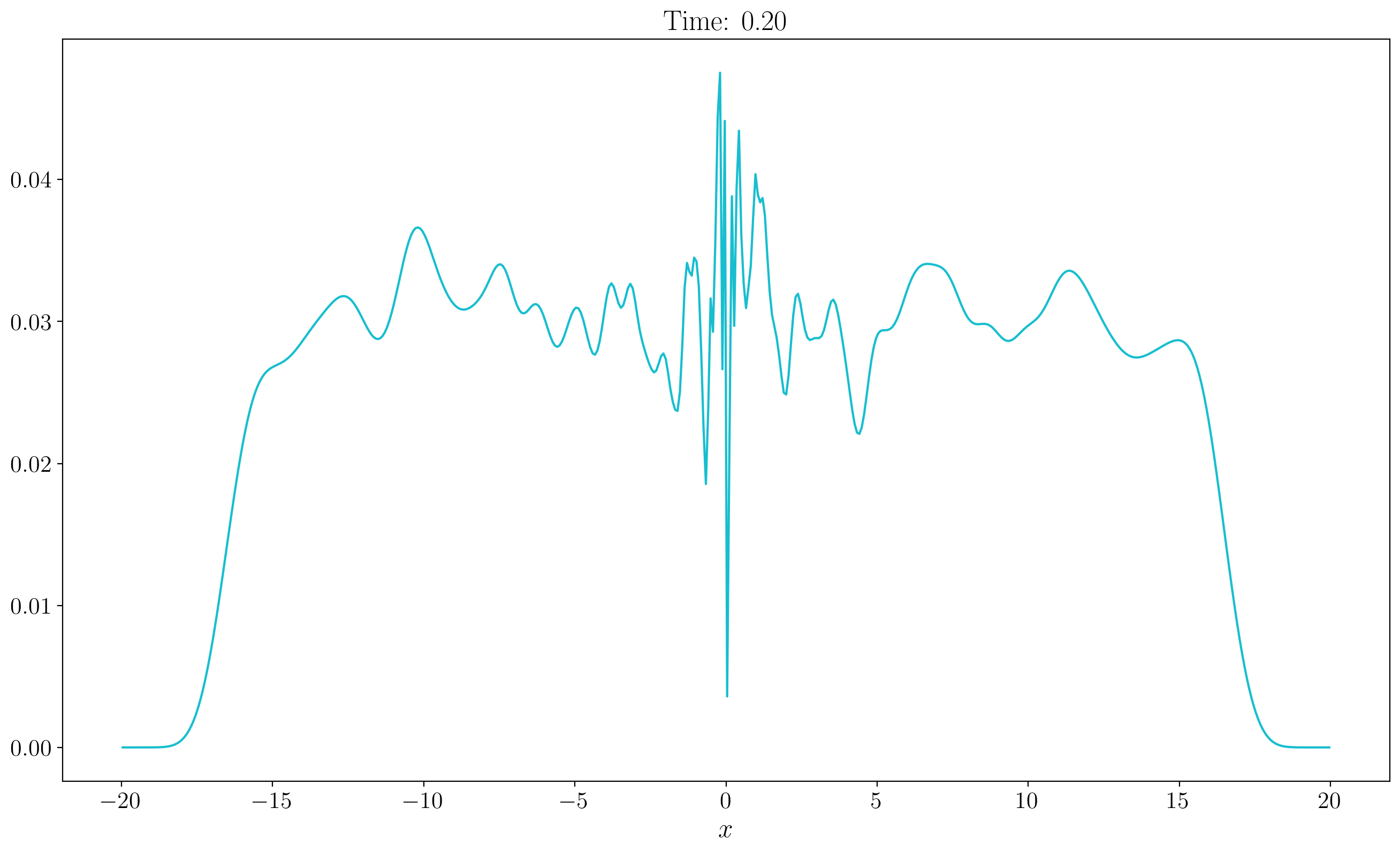} 
\caption{$\varepsilon = 1$}
\end{subfigure}
\caption{Porous medium equation and blob method: $u$ (solution of \eqref{eq:local_PDE_FoPla}) and $\ueps$ (solution of \eqref{eq:nonlocal_PDE_one_dimension_FoPla}) for $\varepsilon=1$. $\Omega = (-20, 20)$,  $u^0(x) = u(x, 0)$ initial datum generated using random values on a uniform grid of $N=2^9$ cells, $\Delta t = 0.01$.}
\label{Fig:RandomProfiles}
\end{figure}

\section*{Acknowledgements}
This work was supported by the Advanced Grant Nonlocal-CPD (Nonlocal PDEs for Complex Particle Dynamics: Phase Transitions, Patterns and Synchronization) of the European Research Council Executive Agency (ERC) under the European Union’s Horizon 2020 research and innovation programme (grant agreement No. 883363). JAC was also partially supported by the EPSRC grant numbers EP/T022132/1 and EP/V051121/1. CE was supported by the ERC
AdG 101054420 EYAWKAJKOS project.

\bibliographystyle{abbrv}
\bibliography{fastlimit}

\begin{thebibliography}{10}

\bibitem{amassad2025deterministic}
A.~Amassad and D.~Zhou.
\newblock A deterministic particle method for the porous media equation.
\newblock {\em arXiv preprint arXiv:2501.18745}, 2025.

\bibitem{MR3050280}
L.~Ambrosio and N.~Gigli.
\newblock A user's guide to optimal transport.
\newblock In {\em Modelling and optimisation of flows on networks}, volume 2062
  of {\em Lecture Notes in Math.}, pages 1--155. Springer, Heidelberg, 2013.

\bibitem{MR2401600}
L.~Ambrosio, N.~Gigli, and G.~Savar\'{e}.
\newblock {\em Gradient flows in metric spaces and in the space of probability
  measures}.
\newblock Lectures in Mathematics ETH Z\"{u}rich. Birkh\"{a}user Verlag, Basel,
  second edition, 2008.

\bibitem{carrillofronzoni2024}
R.~Bailo, J.~A. Carrillo, S.~Fronzoni, and D.~Gómez-Castro.
\newblock A finite-volume scheme for fractional diffusion on bounded domains.
\newblock {\em European Journal of Applied Mathematics}, page 1–21, 2024.

\bibitem{burger2022porous}
M.~Burger and A.~Esposito.
\newblock Porous medium equation and cross-diffusion systems as limit of
  nonlocal interaction.
\newblock {\em Nonlinear Anal.}, 235:Paper No. 113347, 30, 2023.

\bibitem{art_new_formula_Wasserstein}
J.~Carrillo, P.~Gwiazda, and J.~Skrzeczkowski.
\newblock A new semigroup formula for the {W}asserstein distance between two
  {W}asserstein gradient flows.
\newblock {\em In preparation}, 2025.

\bibitem{CarrilloChertock2015}
J.~A. Carrillo, A.~Chertock, and Y.~Huang.
\newblock A finite-volume method for nonlinear nonlocal equations with a
  gradient flow structure.
\newblock {\em Commun. Comput. Phys.}, 17(1):233--258, 2015.

\bibitem{MR3913840}
J.~A. Carrillo, K.~Craig, and F.~S. Patacchini.
\newblock A blob method for diffusion.
\newblock {\em Calc. Var. Partial Differential Equations}, 58(2):Paper No. 53,
  53, 2019.

\bibitem{carrillo2024nonlocal}
J.~A. Carrillo, A.~Esposito, J.~Skrzeczkowski, and J.~S.-H. Wu.
\newblock Nonlocal particle approximation for linear and fast diffusion
  equations.
\newblock {\em arXiv preprint arXiv:2408.02345}, 2024.

\bibitem{carrillo2023nonlocal}
J.~A. Carrillo, A.~Esposito, and J.~S.-H. Wu.
\newblock Nonlocal approximation of nonlinear diffusion equations.
\newblock {\em Calc. Var. Partial Differential Equations}, 63(4):Paper No. 100,
  44, 2024.

\bibitem{MR2935390}
J.~A. Carrillo, L.~C.~F. Ferreira, and J.~C. Precioso.
\newblock A mass-transportation approach to a one dimensional fluid mechanics
  model with nonlocal velocity.
\newblock {\em Adv. Math.}, 231(1):306--327, 2012.

\bibitem{MR4858611}
K.~Craig, M.~Jacobs, and O.~Turanova.
\newblock Nonlocal approximation of slow and fast diffusion.
\newblock {\em J. Differential Equations}, 426:782--852, 2025.

\bibitem{david2023degenerate}
N.~David, T.~Dębiec, M.~Mandal, and M.~Schmidtchen.
\newblock A degenerate cross-diffusion system as the inviscid limit of a
  nonlocal tissue growth model.
\newblock {\em SIAM J. Math. Anal.}, 56(2):2090--2114, 2024.

\bibitem{MR4324293}
N.~David and B.~Perthame.
\newblock Free boundary limit of a tumor growth model with nutrient.
\newblock {\em J. Math. Pures Appl. (9)}, 155:62--82, 2021.

\bibitem{Hecht2023porous}
M.~Doumic, S.~Hecht, B.~Perthame, and D.~Peurichard.
\newblock Multispecies cross-diffusions: from a nonlocal mean-field to a porous
  medium system without self-diffusion.
\newblock {\em J. Differential Equations}, 389:228--256, 2024.

\bibitem{MR4188329}
T.~Dębiec, B.~Perthame, M.~Schmidtchen, and N.~Vauchelet.
\newblock Incompressible limit for a two-species model with coupling through
  {B}rinkman's law in any dimension.
\newblock {\em J. Math. Pures Appl. (9)}, 145:204--239, 2021.

\bibitem{MR4146915}
T.~Dębiec and M.~Schmidtchen.
\newblock Incompressible limit for a two-species tumour model with coupling
  through {B}rinkman's law in one dimension.
\newblock {\em Acta Appl. Math.}, 169:593--611, 2020.

\bibitem{elbar2023inviscid}
C.~Elbar and J.~Skrzeczkowski.
\newblock On the inviscid limit connecting {B}rinkman's and {D}arcy's models of
  tissue growth with nonlinear pressure.
\newblock {\em J. Math. Fluid Mech.}, 27(2):Paper No. 28, 14, 2025.

\bibitem{MR1821479}
P.-L. Lions and S.~Mas-Gallic.
\newblock Une m\'{e}thode particulaire d\'{e}terministe pour des \'{e}quations
  diffusives non lin\'{e}aires.
\newblock {\em C. R. Acad. Sci. Paris S\'{e}r. I Math.}, 332(4):369--376, 2001.

\bibitem{MR3695889}
A.~Mellet, B.~Perthame, and F.~Quir\'{o}s.
\newblock A {H}ele-{S}haw problem for tumor growth.
\newblock {\em J. Funct. Anal.}, 273(10):3061--3093, 2017.

\bibitem{Ol}
K.~Oelschl\"ager.
\newblock Large systems of interacting particles and the porous medium
  equation.
\newblock {\em J. Differential Equations}, 88(2):294--346, 1990.

\bibitem{MR3162474}
B.~Perthame, F.~Quir\'{o}s, and J.~L. V\'{a}zquez.
\newblock The {H}ele-{S}haw asymptotics for mechanical models of tumor growth.
\newblock {\em Arch. Ration. Mech. Anal.}, 212(1):93--127, 2014.

\bibitem{MR3260280}
B.~Perthame, M.~Tang, and N.~Vauchelet.
\newblock Traveling wave solution of the {H}ele-{S}haw model of tumor growth
  with nutrient.
\newblock {\em Math. Models Methods Appl. Sci.}, 24(13):2601--2626, 2014.

\bibitem{Perthame-incompressible-visco}
B.~Perthame and N.~Vauchelet.
\newblock Incompressible limit of a mechanical model of tumour growth with
  viscosity.
\newblock {\em Philos. Trans. Roy. Soc. A}, 373(2050):20140283, 16, 2015.

\bibitem{ref:Santambrogio2015}
F.~Santambrogio.
\newblock {\em Optimal transport for applied mathematicians}, volume~87 of {\em
  Progress in Nonlinear Differential Equations and their Applications}.
\newblock Birkh\"auser/Springer, Cham, 2015.
\newblock Calculus of variations, PDEs, and modeling.

\bibitem{MR3037041}
M.~Tang, N.~Vauchelet, I.~Cheddadi, I.~Vignon-Clementel, D.~Drasdo, and
  B.~Perthame.
\newblock Composite waves for a cell population system modeling tumor growth
  and invasion.
\newblock {\em Chinese Ann. Math. Ser. B}, 34(2):295--318, 2013.

\bibitem{ref:Urbano2008}
J.~M. Urbano.
\newblock {\em The method of intrinsic scaling}, volume 1930 of {\em Lecture
  Notes in Mathematics}.
\newblock Springer-Verlag, Berlin, 2008.
\newblock A systematic approach to regularity for degenerate and singular PDEs.

\bibitem{MR1964483}
C.~Villani.
\newblock {\em Topics in optimal transportation}, volume~58 of {\em Graduate
  Studies in Mathematics}.
\newblock American Mathematical Society, Providence, RI, 2003.

\end{thebibliography}

\end{document}